\documentclass[a4paper, 10pt, reqno, dvipsnames]{amsart}

\usepackage[utf8]{inputenc}

\usepackage{microtype}

\usepackage[english]{babel}

\usepackage{adjustbox}

\usepackage{verbatim}
\usepackage{newcent}

\usepackage{xcolor}

\usepackage{longtable}
\usepackage{graphicx, amsmath, amsfonts, amssymb, amsthm, amscd, amsbsy, amstext, tikz,tikz-cd, mathrsfs, mathtools, enumerate, enumitem, tensor, a4wide,xfrac}   
\usetikzlibrary{decorations.pathreplacing}
\usetikzlibrary{calc}

\usepackage{latexsym,ifthen,stmaryrd,fancyhdr,empheq,verbatim, epigraph}

\makeatletter 
\def\l@subsection{\@tocline{2}{0pt}{2pc}{6pc}{}} \makeatother

\usepackage{capt-of}	

\usepackage{dynkin-diagrams}

\usepackage[colorlinks=true, allcolors=black]{hyperref}

\usepackage{cleveref} 

\usepackage{mathabx}
\usepackage{bbm}
\usepackage{pifont}
\usepackage{manfnt}

\usepackage{tikz-cd}
\usepackage{mathtools}

\makeatletter
\@namedef{subjclassname@2020}{%
	\textup{2020} Mathematics Subject Classification}
\makeatother

\usepackage[OT2, T1]{fontenc} 

\usepackage[textsize=scriptsize]{todonotes}

\usepackage{scalerel, stackengine}

\usepackage{graphicx}

\usepackage{todonotes} 

\usepackage{enumitem}

\mathtoolsset{
	showmanualtags}    

\hypersetup{anchorcolor=black, linktocpage=false,
	colorlinks=true,
	citecolor=RawSienna,
	linkcolor=Mahogany,
	urlcolor=black,
	breaklinks=true,
	plainpages=true
}

\usetikzlibrary{positioning,shapes,shadows,arrows}
\usepackage[mathscr]{euscript} 
\allowdisplaybreaks

\DeclareRobustCommand{\SkipTocEntry}[5]{}

\usepackage[all]{xy}
\usepackage{scalerel}
\usepackage{graphicx}

\usepackage{calligra}

\DeclareMathAlphabet{\mathcalligra}{T1}{calligra}{m}{n}
\DeclareFontShape{T1}{calligra}{m}{n}{<->s*[2.2]callig15}{}

\newtheorem{theorem}{Theorem}
\newtheorem{prop}[theorem]{Proposition}
\newtheorem{lemma}[theorem]{Lemma}
\newtheorem{coro}[theorem]{Corollary}

\newtheorem*{corollary*}{Corollary}
\newtheorem*{theorem*}{Theorem}
\newtheorem*{proposition*}{Proposition}
\newtheorem*{conjecture*}{Conjecture}
\numberwithin{equation}{section}
\numberwithin{theorem}{section}
\newtheorem{example}[theorem]{Example}
\theoremstyle{remark}
\newtheorem{rmk}[theorem]{Remark}
\newtheorem{notation}[theorem]{Notation}
\newtheorem{construction}[theorem]{Construction}

\newtheorem*{strategy*}{Strategy}

\theoremstyle{definition}

\newtheorem*{defi*}{Definition}

\newtheorem{deff}[theorem]{Definition}

\newcommand{\cA}{{\mathcal A}}

\newcommand{\cC}{{\mathcal C}}
\newcommand{\cD}{{\mathcal D}}

\newcommand{\cG}{{\mathcal G}}

\newcommand{\cM}{{\mathcal M}}
\newcommand{\cN}{{\mathcal N}}

\newcommand{\cR}{{\mathcal R}}
\newcommand{\cS}{{\mathcal S}}

\newcommand{\cU}{{\mathcal U}}

\newcommand{\cY}{{\mathcal Y}}

\newcommand{\Hom}{\operatorname{Hom}}

\newcommand{{\bull}}{{\scriptscriptstyle{\bullet}}}

\newcommand{\Left}{\mathsf{Left}}

\newcommand{\Op}{\mathsf{Op}}

\newcommand{\Cat}{\mathsf{Cat}}

\newcommand{\alg}{\mathsf{Alg}}
\newcommand{\env}{\mathsf{Env}}

\newcommand{\omg}{\ensuremath{\mathbf{\Omega}}}

\newcommand{\ssets}{\ensuremath{\mathbf{sSets}}}
\newcommand{\dsets}{\ensuremath{\mathbf{dSets}}}

\DeclareMathOperator{\colim}{\mathsf{colim}}

\newcommand{\id}{\text{id}}

\newcommand{{\op}}{{{\rm op}}}
\newcommand{{\coop}}{{{\rm coop}}}

\usepackage{pict2e}

\makeatletter
\newcommand{\adjunction}[4]{%
	#1\colon #2%
	\mathrel{\vcenter{%
			\offinterlineskip\m@th
			\ialign{%
				\hfil$##$\hfil\cr
				\longrightharpoonup\cr
				\noalign{\kern-.3ex}
				\smallbot\cr
				\longleftharpoondown\cr
			}%
	}}%
	#3 \noloc #4%
}
\newcommand{\longrightharpoonup}{\relbar\joinrel\rightharpoonup}
\newcommand{\longleftharpoondown}{\leftharpoondown\joinrel\relbar}
\newcommand\noloc{%
	\nobreak
	\mspace{6mu plus 1mu}
	{:}
	\nonscript\mkern-\thinmuskip
	\mathpunct{}
	\mspace{2mu}
}
\newcommand{\smallbot}{%
	\begingroup\setlength\unitlength{.15em}%
	\begin{picture}(1,1)
		\roundcap
		\polyline(0,0)(1,0)
		\polyline(0.5,0)(0.5,1)
	\end{picture}%
	\endgroup
}
\makeatother

\begin{document}

\title{Rectification of dendroidal left fibrations} 

\author{Francesca Pratali}

\begin{abstract}
For a discrete colored operad $P$, we construct an adjunction between the category of dendroidal sets over the nerve of $P$ and the category of simplicial $P$-algebras, and prove that when $P$ is $\Sigma$-free it establishes a Quillen equivalence with respect to the covariant model structure on the former category and the projective model structure on the latter. When $P=A$ is a discrete category, this recovers a Quillen equivalence previously established by Heuts-Moerdijk, of which we provide an independent proof. 

\noindent To prove the constructed adjunction is a Quillen equivalence, we show that the left adjoint presents a previously established operadic straightening equivalence between $\infty$-categories. This involves proving that, for a discrete symmetric monoidal category $A$, the Heuts-Moerdijk equivalence is a monoidal equivalence of monoidal Quillen model categories.
\end{abstract}

\address{F.P.: LAGA, Universit\'e Sorbonne Paris Nord, 99 avenue Jean-Baptiste Cl\'ement, 93430 Villetaneuse, France}

\email{pratali@math.univ-paris13.fr}



\maketitle

\tableofcontents

\section*{Introduction and main results}

\subsection*{Introduction}
\noindent Topological operads were introduced by May \cite{M:TGILS} and Boardman-Vogt \cite{BV:HIASTS} to describe the up to homotopy algebraic structures on topological spaces, and since then they have become a standard tool in algebra, geometry, combinatorics, and mathematical physics. 
Intuitively, a topological operad $P$ is the data of a set $C(P)$ of objects, also called colors, and for any choice of objects $c_1,\dots,c_n,c$ in $P$, a space of multimorphisms $P(c_1,\dots,c_n;c)$, together with appropriate partial composition laws  $$\circ_{c_i}\colon P(c_1,\dots,c_n;c)\times P(d_1,\dots,d_m;c_i)\longrightarrow P(c_1,\dots,c_{i-1},d_1,\dots,d_m,c_{i+1},\dots,c_n;c) $$ for any $i\in \{1,\dots,n\}$. 

\noindent In homotopy theory, topological spaces are regarded only up to weak homotopy equivalence, and for this purpose it is more convenient to represent topological spaces via \emph{simplicial sets}, the presheaves on the category of finite linear orders $\Delta$, and to deal with homotopy theories via Quillen model category structures. The diagramatic operadic intuition sees the natural occurence of trees, which can be used to model operadic composition in a combinatorial way: the above partial composition law can be depicted as

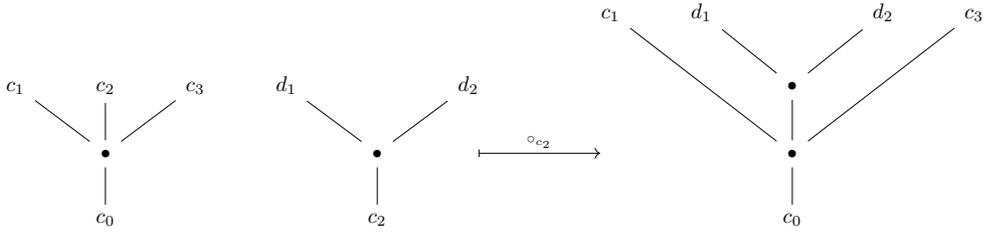
\adjustbox{scale=0.8,center}{%
\begin{tikzcd}
	&                                                &     &                         &                                                &                                       &  & c_1 \arrow[rrdd, no head] & d_1 \arrow[rd, no head] &                                                  & d_2 & c_3 \\
	c_1 \arrow[rd, no head] & c_2 \arrow[d, no head]                         & c_3 & d_1 \arrow[rd, no head] &                                                & d_2                                   &  &                           &                         & \bullet \arrow[ru, no head] \arrow[d, no head]   &     &     \\
	& \bullet \arrow[ru, no head] \arrow[d, no head] &     &                         & \bullet \arrow[ru, no head] \arrow[d, no head] & {} \arrow[rr, "\circ_{c_2}", maps to] &  & {}                        &                         & \bullet \arrow[rruu, no head] \arrow[d, no head] &     &     \\
	& c_0                                            &     &                         & c_2                                            &                                       &  &                           &                         & c_0                                              &     &    
\end{tikzcd}
}

\noindent This intuition has been made precise via the \emph{dendroidal formalism}, introduced by Moerdijk-Weiss in \cite{MW:DS}, whose key construction is that of the \emph{dendroidal category} $\omg$, a category of trees built as an extension of $\Delta$ and meant to capture operadic composition and higher associativities. In particular, topological, or rather, simplicial operads and their homotopy theory are encoded by the category of presheaves on $\omg$, called \emph{dendroidal sets}, and Quillen model structures on it.

\noindent Paralleling the fact that simplicial operads generalize simplicial categories, dendroidal sets and their homotopy theories form a natural extension of those of simplicial sets, essentially because the category $\omg$ extends $\Delta$. It means that there is a Quillen model structure on dendroidal sets, referred to as the operadic model structure, which is Quillen equivalent to a model structure à la Dwyer-Kan on topological operads and which extends the equivalence between Joyal model structure on simplicial sets and the Dwyer-Kan model structure on topological categories (\cite{CM:DSSO}). Several localizations of the operadic model structure are homotopically meaningful and recover non-multiplicative analogues on simplicial sets; an example of this is the covariant model structure for dendroidal left fibrations: meant to homotopically model operads cofibred in groupoinds, it is a central construction of this work, and when restricted to simplicial sets, it yields the covariant model structure for left fibrations obtained as a localization of Joyal's model structure.

\noindent Many topological operads of interest arise as the localization of \emph{discrete} operads, namely operads whose spaces of multimorphisms are discrete topological spaces, that is, sets. It is the case for the little disks operad $\mathbb{E}_n$ and more generally for its variant $\mathbb{E}_M$ constructed out of a topological manifold $M$ of dimension $n$ \cite[\S 5]{Lu:HA}. Indeed, one proves that the $\infty$-operad $\mathbb{E}_M$ is equivalent to the discrete operad of disks on $M$ with disjoint union and inclusion localized as isotopy equivalences (\cite{Lu:HA}, \cite{AFT:FHSS}), which is at the core of the theory of \emph{factorization homology}.

\noindent It is therefore fundamental to understand the homotopy theory of algebras over a discrete operad, and this is what we focus our attention on in this work, where we show that it is equivalent to the homotopy theory of dendroidal left fibrations over the dendroidal nerve of $P$. 

\subsection*{Main results}

\noindent Consider a discrete operad $P$. The category of simplicial $P$-algebras $\alg_P(\ssets)$ has a Quillen model structure, called the \emph{projective model structure}, right transfered from the Kan-Quillen model structure on simplicial sets (\cite{BM:AHTO}): weak equivalences, resp. fibrations, are objectwise weak homotopy equivalences, resp. Kan fibrations. 

\noindent Any discrete operad can be realized as a dendroidal set via the \emph{dendroidal nerve functor}, a fully faithful functor $\cN_d\colon \Op\to\dsets$. Given any dendroidal set $X$, the \emph{covariant model structure} on $\dsets/X$ \cite{He:AOIO} has as the fibrant objects \emph{dendroidal left fibrations} $\alpha\colon E \to X$, making $X$ cofibred in Kan complexes. 
Our main result is the following

\begin{theorem*}[\Cref{QEstrunstr}]
Let $P$ be a $\Sigma$-free discrete operad. There is a natural Quillen equivalence $$\adjunction{\rho_!^P}{\dsets/\cN_d P}{\alg_P(\ssets)}{\rho^*_P} $$ between the covariant model structure on dendroidal sets over the dendroidal nerve of $P$ and the projective model structure on simplicial $P$-algebras.
\end{theorem*}

 \noindent When $X$ is a simplicial set, dendroidal left fibrations over $X$ are precisely the left fibrations of simplicial sets over $X$, and there is an equivalence of covariant model structures ${\dsets/M\simeq \ssets/M}$. When $P=A$ is a discrete category, its dendroidal nerve is the usual nerve, and the adjunction in \Cref{QEstrunstr} coincides with the one constructed by Heuts-Moerdijk in \cite{HM:LFHC}, which in particular means that the right derived functor of $\rho^*_A$, which is equivalent to Lurie's \emph{relative nerve functor} \cite{Lu:HTT}, is equivalent to the left derived functor of the homotopy colimit functor. In addition to this we prove that, when the category $A$ has a symmetric monoidal structure, the equivalence is monoidal. More precisely:
 
 \begin{theorem*}[\Cref{kiku}]
 	Let $A$ be a discrete symmetric monoidal category. The Quillen pair  $$ \adjunction{\rho_!^A}{\ssets/\cN A}{\mathsf{Fun}(A,\ssets)}{\rho^*_A}$$ is a monoidal Quillen equivalence of Quillen monoidal model categories, with respect to the projective model structure and Day convolution on $\mathsf{Fun}(A,\ssets)$ and the covariant model structure and the $\boxtimes$-product (\ref{section6.2}) on $\ssets/\cN A$.
 \end{theorem*}
 
 \noindent We can rephrase the above result by saying that the Quillen equivalence $(\rho_!^A,\rho^*_A)$ gives a presentation of the monoidal un/straightening equivalence established by Ramzi in \cite{R:MGCIC} in the case of a discrete symmetric monoidal category.
 
 \subsection*{Strategy}
\noindent The proof of the main result, \Cref{QEstrunstr} is carried on in two distinct steps, the second of which makes use of \Cref{kiku} in the case the discrete category $A$ is the symmetric monoidal envelope of $P$.

\noindent The first step consists in constructing the adjunction, which we do by defining the left adjoint. Then we prove that the pair obtained is a Quillen adjunction. Consider a discrete operad $P$, non necessarily $\Sigma$-free.

\begin{proposition*}[\Cref{adjj}]
	There exists a pair of adjoint functors $$\adjunction{\rho_!^P}{\dsets/\cN_d P}{\alg_P(\ssets)}{\rho^*_P},$$ and it is a Quillen adjunction between the covariant model structure on dendroidal sets over the nerve of $P$ and the projective model structure on simplicial $P$-algebras.
\end{proposition*}

 \noindent Let us make the following remarks of the proof of \Cref{adjj}:
\begin{itemize}
	\item To construct $\rho_!^P$, we proceed via left Kan extension and impose naturality with respect to base change, so that we reduce to defining a simplicial $T$-algebra $\cA^T=\rho^T(T,\id_T)$ for any tree $T$ in $\omg$, regarded as the operad it generates. Given an edge $e$ of $T$, the simplicial set $\cA^T(e)$ is defined as the nerve of the poset of subtrees of $T$ with root $e$, and the operadic action corresponds to grafting of trees.
	\item We deduce the description of the $P$-algebra $\rho_!^P(X,\alpha)$ when $X\simeq T$ is a tree; with this in hand, we are able to prove that $\rho^*_P$ is right Quillen, which makes the pair $(\rho_!^P, \rho^*_P)$ a Quillen adjunction.
\end{itemize}

\noindent Showing that $(\rho_!^P,\rho^*_P)$ is a Quillen equivalence while still using the language of model categories suggests one needs to describe the simplicial $P$-algebra $\rho_!^P(X,\alpha)$ for any dendroidal set $X$. We start with the following observation.

\begin{proposition*}
Consider a tree $T$ and an element $(T,\alpha)$ in $\dsets/\cN_d P$. On an object $c$ of $P$, the value of the simplicial $P$-algebra $\rho_!^P(T,\alpha)$ can be written as $$\rho_!^P(T,\alpha)(c)\simeq \env(T)\times_{\env(P)} {\env(P)}_{/c} ,$$ where $\env(-)$ is the nerve of the symmetric monoidal envelope of a discrete operad.
\end{proposition*}

\noindent As any dendroidal set can be written as the colimit of the trees mapping into it, and as the rectification functor $\rho_!^P$ is a left adjoint, one has $$ \rho_!^P(X,\alpha)\simeq \colim_{T\to X} \rho_!^P(T,\alpha').$$ However, the above colimit of algebras hard to compute, the main obstacle being the fact that (non sifted) colimits in the category $\alg_P(\ssets)$ are not computed objectwise -a phenomenon we do not observe in functor categories. 

\noindent Instead, we directly study the left derived functor $\mathbb{L}\rho_!^P$ of $\rho_!^P$, or rather the functor of ${\text{{$\infty$-categories}}}$ $(\rho_!^P)_\infty$ induced by $\rho_!^P$ by virtue of the fact that it is a Quillen functor. The key result is the following

\begin{theorem*}[{\Cref{desame}}]
	Let $P$ be a discrete \emph{$\Sigma$-free} operad, there is an equivalence of functors of $\infty$-categories $(\rho_!^P)_\infty\simeq \mathsf{St}^P$, where $\mathsf{St}^P$ is the \emph{operadic straightening functor} of \cite{P:SUEIO}.
	
	\noindent In other words, the left Quillen functor $\rho_!^P$ presents the operadic straightening functor $\mathsf{St}^P$.
\end{theorem*}

\noindent As in \cite{P:SUEIO} the functor $\mathsf{St}^P$ is proved to be an equivalence of $\infty$-categories, we deduce that the left derived functor $\mathbb{L}\rho_!^P$ is an equivalence between the homotopy categories, which means that $(\rho_!^P,\rho^*_P)$ is a Quillen equivalence. 

\noindent Besides, we also deduce an explicit description of the left derived functor $\mathbb{L}\rho_!^P$ on \emph{any} dendroidal left fibration over $\cN_d P$, which we record in the following final

\begin{corollary*}
	Let $(X,\alpha)$ be a dendroidal left fibration in $\dsets/\cN_d P$. On an object $c$ of $P$, the value of the simplicial $P$-algebra $\mathbb{L}\rho_!^P(X,\alpha)$ is weakly homotopy equivalent to the fibre product $$\mathbb{L}\rho_!^P(X,\alpha)(c)\simeq \env(X)\times_{\env(P)} {\env(P)}_{/c} ,$$ where $\env(X)$ denotes the underlying $\infty$-category of Lurie's symmetric monoidal envelope of (the Lurie $\infty$-operad equivalent to) $X$.
\end{corollary*}

\subsection*{Outline} \hfill 

\begin{enumerate}
\item In \Cref{section1} we recall the fundamentals of the dendroidal formalism, the notions of quasioperad, dendroidal left fibration and the covariant model structure. We also set up notation.

\item In \Cref{section4}, we construct the rectification functor $\rho_!^P$ (\Cref{sec1subs1}) and then give an explicit description of its right adjoint $\rho^*_P$ (\Cref{sec1subsec2}). 

\item In \Cref{section5}, we prove that the pair $(\rho_!^P,\rho^*_P)$ forms a Quillen adjunction by showing that $\rho^*_P$ is right Quillen (\Cref{adjj}). 
More precisely, in \Cref{matija} we relate root preserving dendroidal faces of a tree with chains in the linear order of posets of the tree with a fixed root. We employ this construction in the next two sections, where we give necessary conditions to produce lifts of morphisms of the form $\rho^*_P(F)\to\rho^*_P(G) $ against dendroidal boundary inclusions (\Cref{reducelift}) and inner and left horn inclusions (\Cref{projfibleftfib}).

\item In \Cref{section6}, we focus on the case where $P=A$ is a discrete category. First, we show that in this case the Quillen adjunction is a Quillen equivalence (\Cref{bonobo}). Then we prove that, when $A$ has also a symmetric monoidal structure, both model categories in \Cref{adjj} are \emph{monoidal} model categories, and we prove that the Quillen equivalence is a monoidal equivalence of Quillen model categories (\Cref{kiku}).

\item In \Cref{section8} we come back to the case where $P$ is a general discrete operad. First, we recall the key facts needed to interpret the model categorical statements as statements about $\infty$-categories and rephrase in this language the results of the previous sections (\Cref{sofar}). In \Cref{panico} we recall the definition of the operadic straightening functor, and in \Cref{fine} we prove that, under the assumption that $P$ is moreover $\Sigma$-free, the rectification functor presents the straightening functor (\Cref{desame}), deducing the Quillen equivalence.
\end{enumerate}


\subsection*{Acknowledgements} I would like to thank Gijs Heuts for suggesting this problem to me and for the helpful discussions. I would like to thank Joost Nuiten, Miguel Barata, Hugo Pourcelot, Victor Saunier and Maxime Ramzi for interesting comments, Ieke Moerdijk for insightful exchanges and Matija Bašić for the help with some of the combinatorics in \Cref{section5}. I would like to thank my advisor, Eric Hoffbeck, for his constant support and interest in my projects, and for carefully reading my drafts.

\noindent Part of this work originated while visiting the University of Utrecht, in 2024, cofinanced by the Eole grant (RFN). My PhD is founded by the European Union’s Horizon 2020 research and innovation programme under the Marie Skłodowska-Curie grant agreement No 945332.

\section{Recollections: dendroidal formalism}\label{section1}
\noindent Let us start by recalling some facts and set up some notation regarding colored operads and dendroidal sets. Complete proofs and definitions can be found in \cite{HeMo:SDHT}.

\subsection{Operads and their algebras}\label{dalg}\hfill

\noindent We denote by $\Op$ the category of discrete colored operads, from now on just called operads. Given an operad $P$, we write $C(P)$ its set of objects and $P(c_1,\dots,c_n;c)$ for its set of $n$-ary operations with input objects $(c_1,\dots,c_n)$ and output $c$. A discrete category $C$ is an operad that only has one-input one-output operations, so the category of discrete categories $\Cat$ is a full subcategory of operads. The inclusion $\Cat\subseteq \Op$ has a right adjoint, which associates to a given operad its \emph{underlying category}, where one only remembers unary operations of  $P$. We single out the following special class of operads.

\begin{deff}
An operad $P$ is \emph{$\Sigma$-free} if, for any natural number $n$ and any choice of objects $c_1,\dots, c_n, c$, the symmetric group on $n$-elements $\Sigma_n$ acts freely on the set $	\underset{{\sigma \in \Sigma_n}}{\bigcup} P(c_{\sigma(1)},\dots,c_{\sigma(n)};c).$
\end{deff}
\noindent For instance, any category is tautologically $\Sigma$-free. The associative operad is $\Sigma$-free, while the commutative operad it is not.

\noindent A \emph{simplicial $P$-algebra} $F$ is given by a family of simplicial sets $\{F(c)\}_{c\in C(P)}$ and action maps $p_*\colon F_{c_1}\times \dots \times F_{c_n}\to F_c$ for any $p\in \mathcal{P}(c_1,\dots,c_n;c)$, compatible with operadic composition and the symmetric group action. A morphism $\varphi\colon F \to G$ of simplicial $P$-algebras is a collection of maps of simplicial sets $\varphi_c\colon F(c)\to G(c)$, one for every object $c$ of $P$, satisfying the same compatibility conditions. We denote by $\alg_P(\ssets)$ the category of simplicial algebras over $P$.

\begin{rmk}
Algebras on a category $A$ are just functors: there is an identification $\alg_A(\ssets)= \mathsf{Fun}(A,\ssets)$.
\end{rmk}

\noindent Given a morphism of operads $\varphi\colon P \to Q$, there is an induced adjunction between the categories of algebras $\varphi$, $$\adjunction{\varphi_!}{\alg_P(\ssets)}{\alg_Q(\ssets)}{\varphi^*},$$ which we call \emph{base change}. Given $F\in \alg_Q(\ssets)$ and a object $c$ of $P$, one has $$\varphi^*(F)(c)= F({\varphi(c)}), $$ and for $p\in P(c_1,\dots,c_n;c)$, the action map $$p_*\colon \varphi(F)^*({c_1})\times \dots \times \varphi(F)^*({c_1})\longrightarrow \varphi^*(F)(c) $$ is the map $$ \varphi(p)_*\colon F({\varphi(c_1)})\times \dots\times F({\varphi(c_n)})\longrightarrow F({\varphi(c)}).$$ The left adjoint $\varphi_!$ sends free $P$-algebras to free $Q$-algebras, and an explicit description can be found in \cite[\S 4]{BM:RCORHA}.

\subsection{Dendroidal sets}\hfill

\noindent The category of dendroidal sets $\dsets$ is the category of presheaves on the tree category $\omg$. As the simplex category $\Delta$ can be defined as a full subcategory of the category of discrete categories, the dendroidal category $\omg$ can be defined as a full subcategory of the category of discrete operads $\Op$; let us explain how.

\subsubsection{The category $\omg$} The objects of $\omg$ are non-planar trees $T$ with finite vertex set $V(T)$ and edge set $E(T)$, together with a specified edge, denoted by $r_T$, which is attached to a single vertex. We call this edge the \emph{root} of $T$ and this vertex the \emph{root vertex}.

\adjustbox{scale=0.9,center}{%
	\begin{tikzcd}
& && {}&&&&&&&\\
&& {} \arrow[rd, no head] & \bullet \arrow[u, no head] &{}&&{} \arrow[rd, no head]&{}&{}&& {} \arrow[dd, no head] \\
T = & {} & {} \arrow[rd, no head] & \bullet \arrow[ru, no head] \arrow[d, no head] \arrow[u, no head]& \bullet & {} & C_3 \ =& \ \ \bullet \ v \arrow[ru, no head] \arrow[u, no head] \arrow[d, no head] & & \eta \ = &\\
& && \bullet \arrow[d, "r_T", no head] \arrow[ru, no head] \arrow[llu, no head] \arrow[rru, no head] &&&                        & {}&&&{} \\
&   &                        & {}                                                                                              &         &        &                        &                                                                   &     &          &                       
\end{tikzcd}
}
\vspace*{-.3cm}

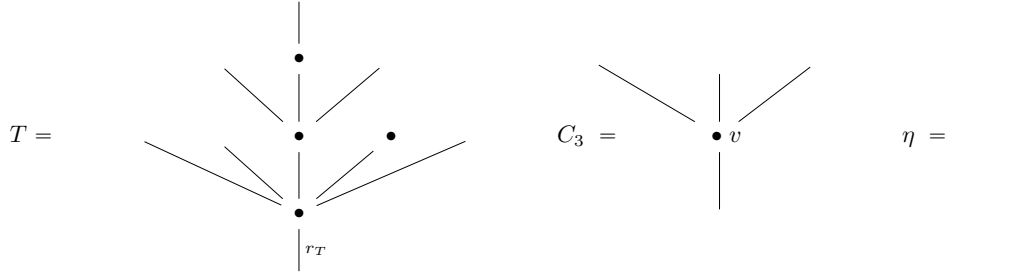
\captionof{figure}{Some typical trees in $\omg$.}

\noindent Given a tree $T$, every vertex $v$ has a single output edge and $n$ input edges, with $n\geq 0$.
The \emph{leaves} of $T$ are the edges which are not the output edge of a vertex, and a vertex whose set of input edges is non-empty and contained in the set of leaves of $T$ is called a \emph{leaf vertex}. An edge $e$ of $T$ is an \emph{inner edge} if it is both the input edge of two (necessarily distinct) vertices.

\noindent A \emph{subtree} of a tree $T$ is a smaller tree $S$ contained in $T$ that satisfies the following condition: if a vertex $v$ of $T$ is contained in $S$, then so are all input edges of $v$. In particular, a subtree of $T$ is uniquely specified by listing its vertices; conversely, a collection of vertices of $T$ defines a subtree only if the graph consisting of all those vertices together with the edges attached to them is connected.

\noindent To any tree $T$ one can associate a free operad $\Omega(T)$, whose objects are the edges of $T$ and where, for edges $e_1,\dots,e_n, e$, one has $ \Omega(T)(e_1,\dots,e_n;e)=\{\ast\}$ if there exists (and if it does it is unique) a subtree of $T$ with leaves $\{e_1,\dots,e_n\}$ and root $e$ and is empty otherwise; the operadic composition corresponds to grafting of subtrees (see \Cref{botanic}). By definition, a morphism of trees $S \to T$ in the dendroidal category $\omg$ is a morphism of discrete operads $\Omega(S)\to \Omega(T)$, and this is how one realizes the inclusion $\Omega(-)\colon \omg\hookrightarrow \Op$. 

\subsubsection{Botanical notions}\label{botanic} The set of edges $E(T)$ has a partial order: given edges $e$ and $f$, one sets $$ e \leq f \ \text{ if and only if the (unique) path from $e$ to the root meets $f$.} $$ In other words, one has $e\leq f$ if and only if $f$ is 'below' $e$. In particular, the root is the unique maximal element of $E(T)$, while the minimal elements are the leaves of $T$ and the output edges of nullary vertices. We call a tree \emph{linear} if all its vertices have valence $1$; in particular, in a linear tree the partial order on the set of edges is a total order.
\noindent Via the output function $\mathsf{out}\colon V(T)\to E(T)$, which assigns to a vertex its output edge, the set $V(T)$ of vertices of $T$ is endowed with a partial order, where one has $v\leq w$ if and only if $\mathsf{out}(v)\leq \mathsf{out}(w)$.  In particular, any map of trees $f\colon S\to T$ induces maps of posets $E(S)\to E(T)$, $V(S)\to V(T)$; however, unless the trees are linear, a map of posets does not necessarily determine a map of trees, see \cite[Proposition 3.5]{HeMo:SDHT}.

\noindent The \emph{grafting} of trees consists in obtaining a new tree $T$ by identifying the root $r_T$ of a tree $R$ with a leaf $l$ of another tree $S$. Formally, $T$ is obtained as the pushout in $\omg$
\[
\begin{tikzcd}
\eta \arrow[r, "l"] \arrow[d, "r_R"] & S \arrow[d] \\
R \arrow[r] & T \ ,
\end{tikzcd}
\]

\noindent and pictorially $T$ can be depicted as follows: \hspace{-5cm}\adjustbox{scale=0.6,center}
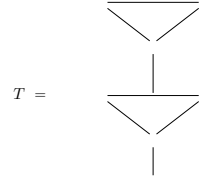
{
\begin{tikzcd}
	& {} \arrow[rr, no head]                     &                                                               & {} \\
	&                                            & {} \arrow[d, no head] \arrow[lu, no head] \arrow[ru, no head] &    \\
	T \ =  & {} \arrow[rd, no head] \arrow[rr, no head] & {}                                                            & {} \\
	&                                            & {} \arrow[ru, no head] \arrow[d, no head]                     &    \\
	&                                            & {}                                                            &   
\end{tikzcd}
}

\noindent If $l_1,\dots l_n$ are leaves of $S$ and $T_1,\dots,T_n$ are trees, one can graft each $T_i$ onto $l_i$ and obtain a tree $S\underset{l_1,\dots,l_m}{\circ}(T_1,\dots,T_n)$. Observe that the corollas are the indecomposable objects for the grafting, in the sense that, if a corolla $C_n$ is written as the grafting of $R$ along $S$, then either $R$ or $S$ are trivial (i.e. $\simeq \eta$).

\noindent As it happens for $\Delta$, morphisms in $\omg$ can be described in a combinatorial way. We can distinguish four classes of morphisms in $\omg$:
\begin{itemize}
	\item isomorphisms  $T\xrightarrow{\sim} T'$. Observe that, contrarily to $\Delta$, these can be non-trivial;
	\item for every inner edge $e$ of $T$, we call \emph{elementary inner face} the map $\partial_eT\to T$ which comes from contracting $e$ in $T$ and identifying its extremal vertices. If $R \to T$ is obtained by contracting more than one inner edge, we call it \emph{inner face}.
	\item for every subtree $S$ of $T$, there is the \emph{external face} consisting in the inclusion $S\hookrightarrow T$;
	\item  for every edge $e$ of $T$, there is a \emph{degeneracy} $\sigma_e T \to T$ which adds a unary vertex in the middle of $e$.
\end{itemize}
\noindent We call \emph{face map} the composition of inner and external faces and a \emph{degeneracy} the composition of degeneracies. We say that an external face is \emph{elementary face map} if it adds precisely one vertex, and more generally we call \emph{elementary faces} those face maps which add or erase precisely one vertex. Restricted to $\Delta\subseteq \omg$, one recovers the usual face and degeneracy morphisms.

\noindent Face maps, degeneracies and isomorphisms generate the morphisms in $\omg$; these maps satisfy some relations called \emph{dendroidal identities} (see \cite[Sec 3.3.4]{HeMo:SDHT}).

\begin{example} Intuitively, we can think about the generating morphisms of the category $\omg$ in the following informal terms 
	\begin{itemize}
		\item an \emph{inner face} corresponds to sending an operation $p$ to an operadic composition $q\circ_i r$, where $q$ and $r$ are unknown.
		
		\adjustbox{scale=0.6, center}
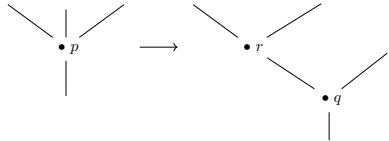
{
			\begin{tikzcd}
				{} \arrow[rd, no head] & {}                                                                      & {}           & {} \arrow[rd, no head] &                                                       & {}                                                   &    \\
				& \ \bullet \ p \arrow[u, no head] \arrow[ru, no head] \arrow[d, no head] & {} \arrow[r] & {}                     & \ \bullet \ r \arrow[ru, no head] \arrow[rd, no head] &                                                      & {} \\
				& {}                                                                      &              &                        &                                                       & \ \bullet \ q \arrow[d, no head] \arrow[ru, no head] &    \\
				&                                                                         &              &                        &                                                       & {}                                                   &   
		\end{tikzcd}}
		\item an \emph{external face} is an inclusion of operads.
		
		\adjustbox{scale=0.6,center}
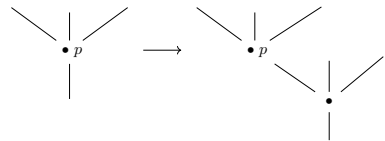
{
	\begin{tikzcd}
		{} \arrow[rd, no head] & {}                                                                      & {}           & {} \arrow[rd, no head] & {} \arrow[d, no head]                                 & {}                                                                     &    \\
		& \ \bullet \ p \arrow[u, no head] \arrow[ru, no head] \arrow[d, no head] & {} \arrow[r] & {}                     & \ \bullet \ p \arrow[ru, no head] \arrow[rd, no head] & {}                                                                     & {} \\
		& {}                                                                      &              &                        &                                                       & \ \bullet \  \arrow[d, no head] \arrow[ru, no head] \arrow[u, no head] &    \\
		&                                                                         &              &                        &                                                       & {}                                                                     &   
		\end{tikzcd}}
		\item  a \emph{degeneracy} sends a unary morphism to the identity.
		
		\adjustbox{scale=0.6,center}
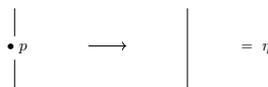
{
			\begin{tikzcd}
				{}                                                  &              &    & {} \arrow[dd, no head] &          \\
				\ \bullet \ p \arrow[u, no head] \arrow[d, no head] & {} \arrow[r] & {} &                        & = \ \eta \\
				{}                                                  &              &    & {}                     &         
		\end{tikzcd}}
	\end{itemize}
\end{example}

\noindent Some of the notions of the homotopy theory of simplicial sets can be formulated in the context of dendroidal sets. In particular, we recall the following

\begin{deff} 
	Let $T$ be a tree.
	\begin{itemize}
		\item Its \emph{boundary inclusion} $\partial T \to \Omega[T]$ is the morphism of dendroidal sets given by the union of all the faces of $T$. 
		\item For every inner edge $e$ of $T$, the \emph{inner horn inclusion} $\Lambda^eT\to \Omega[T]$ is given by the union of all elementary faces except the inner $\partial_eT$.
		\item A vertex $v$ of $T$ is a \emph{leaf vertex} if its input edges are leaves of $T$. For any such $v$, there is an induced \emph{leaf horn incusion} $\Lambda^vT\to \Omega[T]$, defined as the union of all elementary faces except the external $\partial_v T$. 
		
		\noindent If $T\simeq C_n$ is a corolla with a unique vertex $v$, we interpret $\Lambda^v C_n \eqqcolon \ell(C_n)$ as the disjoint union of the leaves of $C_n$, and write $\ell(C_n)\to \Omega[C_n]$ for the corresponding leaf horn inclusion.
	\end{itemize} 
\end{deff}

\subsection{Dendroidal homotopy theory}

\subsubsection{Dendroidal nerve functor} The fully faithful functor $\omg\to \Op$ induces by left Kan extension an embedding $\Op\to\dsets$, whose right adjoint we call the \emph{dendroidal nerve functor} $\cN_d\colon \Op\to \dsets$. For $P\in \Op$ and $T\in \omg$, one has $$\cN_d(P)=\mathsf{Hom}_\Op(\Omega(T),P).$$

\noindent Let $\Delta$ be the skeleton of finite linearly ordered sets. There is a fully faithful functor $i\colon \Delta\to \omg$, realized by sending the linear order $[n]$ to the linear tree with $n+1$ edges and $n$ vertices. One has the identifications $ \eta \simeq i([0])$, $C_1\simeq i([1])$. By left Kan extension, the inclusion induces a fully faithful functor $$ i_!\colon \ssets\to \dsets,$$ and this allows to identify $\ssets$ with a full subcategory of $\dsets$.

\noindent The restriction of the dendroidal nerve to $\Cat$ coincides with the usual nerve functor for categories, meaning there is a commutative diagram 
\[
\begin{tikzcd}
	\Cat \arrow[r, "\cN"]\arrow[d]& \ssets \arrow[d, "i_!"] \\
	\Op \arrow[r, "\cN_d"]& \dsets \ .
\end{tikzcd}
\]

\begin{rmk}
	Denote by $\Omega[T]$ the image of a tree $T$ in $\dsets$ under the Yoneda embedding. There is an isomorphism of dendroidal sets $\Omega[T]\simeq \cN_d( \Omega(T))$.
	In particular, $\Omega[\eta]\simeq i_!(\Delta^0)$, $\Omega[C_1]\simeq i_!(\Delta^1)$ and more generally $\Omega([n])\simeq i_!(\Delta^n)$.
	\end{rmk}

\noindent There is a canonical isomorphism $\dsets/\Omega[\eta]\simeq \ssets$, under which the inclusion of simplicial sets in dendroidal sets is the forgetful functor. More generally, for any simplicial set $M$ there is a canonical isomorphism $\dsets/i_!(M )\simeq \ssets/M$.

\subsubsection{Simplicial enrichment}
\noindent Consider a dendroidal set $X$ and the over category $\dsets/X$. Thanks to the Boardman-Vogt tensor product, dendroidal sets are enriched over simplicial sets \cite{HeMo:SDHT}, and this enrichment passes to $\dsets/X$. 

\noindent Given dendroidal sets $E,B$ denote by $\hom(E,B)$ the simplicial set of morphisms of dendroidal sets from $E$ to $B$; given two elements  $(E,f), (B,g)$ in $\dsets/X$, the simplicial set $\hom_{X}(E,B)$ is defined as the pullback: 
\begin{center}
	\begin{tikzcd}
		{\hom_{X}(E,B)} \arrow[d] \arrow[r]& {\hom(E,B)} \arrow[d, "g_*"] \\
		{\Delta^0} \arrow[r, "\{f\}"]         & {\hom(E, X)}            
	\end{tikzcd}
\end{center} 

\subsubsection{Dendroidal $\infty$-operads: quasioperads} Quasioperads model operads where composition is only defined up-to-homotopy, and are defined as the dendroidal sets satisfying an inner horn filling condition shaped on trees.

\begin{deff}
A dendroidal set $X$ is a \emph{quasioperad} if, for any tree $T$, any inner horn inclusion $\Lambda^eT\to \Omega[T]$ and any map $\Lambda^eT\to X$, the solid diagram below admits a dotted lift:
\begin{center}
\begin{tikzcd}
\Lambda^eT \arrow[r] \arrow[d] & X \\
\Omega[T] \arrow[ru, dashed] &
\end{tikzcd}
\end{center}
\end{deff}

\begin{rmk}
If $M\in \ssets\xhookrightarrow{i_!} \dsets$, then $i_!(M)$ is a quasioperad if and only if $M$ is a quasicategory.
\end{rmk}

\begin{deff}
A dendroidal set $X$ is \emph{normal} if, for any tree $T$, the action on $X_T$ of the group of automorphism of $T$ is free.
More generally, a monomorphism of dendroidal sets $Y\to X$ is \emph{normal} if, for any tree $T$, its group of automorphisms acts freely on $X_T\setminus Y_T$.
\end{deff}

\begin{rmk} If $X\simeq \cN_d (P)$ for a discrete operad $P$, then $X$ is normal if and only if $P$ is $\Sigma$-free.
\end{rmk}

\begin{theorem}[\cite{CM:DSMHO}] There exists a model structure on $\dsets$, called the \emph{operadic model structure}, with the following properties:
	\begin{itemize}
		\item the fibrant objects are the quasioperads;
		\item the cofibrations are normal monomorphisms;
		\item weak equivalences between quasioperads are the homotopically fully faithful and essentially surjective maps;  
		\item the model structure induced on $\dsets/\Omega[\eta]\simeq \ssets$ coincides with the Joyal model structure.
	\end{itemize} 
\end{theorem}

\subsubsection{The covariant model structure}\hfill

\begin{deff}
A \emph{dendroidal left fibration} is a morphism $E\to X$ of dendroidal sets having the right lifting property against inner horn and leaf horn inclusions.
\end{deff}
\noindent In particular, if $X$ is a quasioperad then $E$ is as well, and whenever one chooses an operation of $X$ and preimages in $E$ of the inputs elements, one can lift the chosen operation to $E$ along $E\to X$, and this is can be done homotopy-coherently. Indeed, one can consider the leaf horn inclusion of the leaves of a $n$-corolla $\ell(C_n)\to \Omega[C_n]$, and take a commutative diagram
\[
\begin{tikzcd}
\ell(C_n) \arrow[r] \arrow[d] & E \arrow[d, "p"] \\
\Omega[C_n] \arrow[r, "f"] & X
\end{tikzcd} 
\]
The map $\ell(C_n)\to E$ selects objects $y_1,\dots,y_n$ of $E$, while $f$ corresponds to an $n$-ary operation $f\in X(p(y_1),\dots,p(y_n); x)$ for some $x\in X_\eta$. A lift in the diagram is the existence of a object $y\in E_\eta$ and $\overline{f}\in E(y_1,\dots,y_n; y)$  with $p(\overline{f})=f$, and coCartesianity of the lift is given by higher horns $\Lambda^x T \to \Omega[T]$.

\begin{rmk} If $p\colon E\to X$ is a map of simplicial sets, then $p$ is a dendroidal left fibration if and only if it is a left fibration of simplicial sets.
\end{rmk}

\begin{theorem}[\cite{He:AOIO},\cite{HeMo:SDHT}]\label{covmambo} Let $X$ be a dendroidal set. The category $\dsets/X$ carries a left proper, cofibrantly generated model structure, called the \emph{covariant model structure}, with the following properties:
\begin{enumerate}
\item The cofibrations are the normal monomorphisms over $X$.
\item The fibrant objects are the dendroidal left fibrations $E\to X$.
\item The fibrations between fibrant objects are the dendroidal left fibrations.
\item A map 
\begin{center}
\begin{tikzcd}
E \arrow[rr, "f"] \arrow[rd]& & B\arrow[ld]\\
& X &  
\end{tikzcd}
\end{center} between dendroidal left fibrations $E\to X$ and $B\to X$ is a weak equivalence if and only if, for any object $c\in X_\eta$, the map $ E_c\to B_c$ between fibres over $c$ is a  weak homotopy equivalence of Kan complexes.
\end{enumerate}

\noindent Moreover, this model structure is a left Bousfield localization of the model structure on $\dsets/X$ induced by the operadic model structure on $\dsets$.
\end{theorem}

\noindent The \emph{base change} along a morphism of dendroidal sets $f\colon X \to Y$ consists in the induced adjunction $$ \adjunction{f_!}{\dsets/X}{\dsets/Y}{f^*},$$ where $f_!$ is the functor that composes a morphism $A\to X$ with $f$ and $f^*$ is the functor that takes the pullback along $f$. After \cite{He:AOIO}, the base change is a Quillen adjunction with respect to the covariant model structure, and a Quillen equivalence if $f$ is a weak equivalence in the model structure for quasioperads.

\subsubsection{The projective model structure on simplicial algebras} Given a discrete operad $P$, we can endow the category of simplicial $P$-algebras with a model structure identifying simplicial algebras up to weak homotopy equivalence of spaces.

\noindent Call a morphism $\varphi\colon F\to G$ of simplicial $P$-algebras a \emph{projective weak equivalence}, resp. \emph{projective fibration}, if, for every object $c$ of $P$, the map $\varphi_c\colon F(c)\to G(c)$ is a weak homotopy equivalence, resp. a Kan fibration, of  simplicial sets.

\begin{theorem}[\cite{BM:RCORHA}] Projective weak equivalences and projective fibrations are the weak equivalences and the fibrations of a model category structure on $\alg_P(\ssets)$.
\end{theorem}

\noindent The above model structure is usually refered to as the \emph{projective model structure}.

\subsection{Notation and convention}\label{rememberrr}

\noindent As $i\colon \Delta\to \omg$ is fully faithful, we identify $\Delta$ with a full subcategory of $\omg$ and drop the $i$.

\noindent Similarly, as $i_!\colon \ssets \to \dsets$ is fully faithful, we identify $\ssets$ with a full subcategory of $\dsets$ and drop the $i_!$.

\noindent We identify $\omg$ with a full subcategory of $\Op$, and, as the Yoneda embedding is fully faithful, we identify $\omg$ with a full subcategory of $\dsets$. 

\noindent Given an edge $e$ of $T$, write $T_e^{\uparrow}$ for the biggest subtree of $T$ having $e$ as root. Observe that its set of leaves is contained in that of $T$.

\noindent For edges $\underline{e}=(e_1,\dots,e_n)$, we write $T_e^{\underline{e}}$ for the subtree of $T_e^{\uparrow}$ whose root is $e$ and whose leaves are precisely $\underline{e}$, if it exists. 
The subtree $T^{\underline{e}}_e$ does not depend on the ordering of the tuple $\underline{e}$.
\noindent When it exists, we identify $T^{\underline{e}}_e$ with the corresponding $n$-ary operation of the operad $T$. In particular, given a $T$-algebra $F$, we will write $$(T_e^{\underline{e}})_*\colon F({e_1})\times \dots \times F({e_n})\to F(e)$$ for its action on $F$. 

\noindent To sum up, a tree $T$ can denote:	\begin{itemize}
		\item the element $T$ in $\omg$;
		\item the discrete operad $\Omega(T)$;
		\item the representable dendroidal set $\Omega[T]$;
		\item an operation in a tree $R$, with $T\subseteq R$.
	\end{itemize}

\noindent For coherence, we write $\eta$, resp. $C_1$, for the more familiar $\Delta^0$, resp. $\Delta^1$, and stress the identifications $\eta\simeq 	\Delta^0$, $C_1\simeq \Delta^1$ when needed.
\section{Construction of the adjunction}\label{section4}
\noindent Fix a discrete colored operad $P$. In this section, we will define an adjunction, natural in $P$, of the form $$ \adjunction{\rho^P_!}{\dsets/\cN_d P}{\mathsf{Alg}_P(\ssets)}{\rho_P^*},$$ and we give explicit descriptions of both the left and right adjoint. 

\subsection{Definition of the left adjoint}\label{sec1subs1}

\noindent Any slice category of a presheaf category is itself a presheaf category: in particular, for any dendroidal set $X$ there is an isomorphism $$\dsets/X\simeq \mathsf{Fun}(\left(\omg/X\right)^\text{op}, \mathsf{Set}).$$ The category $\omg/X$ is the category of elements of $X$: its objects are of the form $(T,\alpha)$, where $T\in \omg$ and $\alpha\colon T \to X$ is a morphism of dendroidal sets, and a morphism $f\colon (T,\alpha)\to (S,\beta)$ is a map of trees $f\colon T \to S$ making the obvious triangle commute. 

\noindent When $X\simeq \cN_d P$, fully faithfulness of the nerve ensures that a morphism of dendroidal sets $T\to \cN_d P$ is just a morphism of operads $T\to P$, so let us write simply $\omg/P$  for $\omg/\cN_d P$, and refer to the objects in $\omg/P\xhookrightarrow{\cY} \dsets/\cN_d P$ as the \emph{representables} of $\dsets/\cN_d P$.

\begin{strategy*}
	We define $\rho_!^P$ as the left Kan extension of a functor $\rho^P\colon (\omg/\cN_d P)^\text{op}\to \alg_P(\ssets)$, where $\omg/\cN_d P$ is the category of elements of the presheaf $\cN_d P$. To define the functor $\rho^P$, we require that the family of functors ${\rho_!^P}_P$ is natural with respect to base change of operads, which means that, for any map of operads $\varphi\colon P \to Q$, the following diagram commutes:
	
	\begin{equation}\label{leftadj}
		\begin{tikzcd}
			\dsets/\cN_d P \arrow[r, "\rho^P_!"] \arrow[d, "(\cN_df)_!"'] & \mathsf{Alg}_P(\ssets) \arrow[d, "f_!"] \\
			\dsets/\cN_dQ \arrow[r, "\rho^Q_!"]                               & \mathsf{Alg}_Q(\ssets)     \ .            
		\end{tikzcd}
	\end{equation}
	
	\noindent These two requirements reduce the definition of $\rho_!^P$ to the construction of simplicial $T$-algebras $\cA^T$ for any tree $T$ and functorially in $T$, and it goes as follows.
\end{strategy*}

\begin{construction}\label{sirene}
Fix a tree $T$. For any edge $e$ of $T$, let $P(T_e^{\uparrow})$ be the set of subtrees of $T_e^{\uparrow}$ having $e$ as root. The elements of $P(T_e^{\uparrow})$ are of the form $T^{\underline{e}}_e$ for edges $\underline{e}=(e_1,\dots,e_n)$ for which $ T(e_1,\dots,e_n;e)\neq \emptyset$.  We endow $P(T_e^{\uparrow})$ with the partial order given by reversed inclusion of subtrees: there is an arrow $R\to S$ in $P(T_e^\uparrow)$ if and only if $S\subseteq R$.

\noindent Let $e$ be an edge of $T$, and define $\cA^T(e)$ as the nerve of the poset $P(T_e^{\uparrow})$, i.e. $$ \cA^T(e)\coloneqq \cN(P(T^{\uparrow}_e)).$$ For any choice of edges $e_1,\dots,e_n,e$ of $T$ for which $ T(e_1,\dots,e_n;e)\neq \emptyset$, we define the morphism 
$$\cA^T({e_1})\times\dots\times \cA^T({e_n})\longrightarrow \cA^T(e)$$ as the nerve of the map of posets
$$ P(T_{e_1}^{\uparrow})\times \dots \times P(T_{e_n}^{\uparrow})\longrightarrow P(T_e^{\uparrow}) \ \colon  (R_1,\dots,R_n)\mapsto T^{\underline{e}}_e\circ_{\underline{e}} (R_1,\dots,R_n),$$ 
which sends a $n$-tuple of subtrees $(R_1,\dots,R_n)$, with respective roots $\underline{e}=(e_1,\dots,e_n)$, to the tree obtained as the grafting of $(R_1,\dots,R_n)$ onto $T^{\underline{e}}_e$. Since operadic composition in $T$ is precisely the grafting of trees and a map of trees sends subtrees to subtrees, respecting the inclusion relation, this endows the family $\cA^T\coloneqq \{\cA^T(e)\}_e$ with a simplicial $T$-algebra structure. 

\noindent The construction is natural in $T$ in the following sense. Given a map of trees $f\colon S \to T$ and an edge $e$ of $S$, there is a morphism of simplicial sets $f^\cA_e\colon \cA^S(e)\to \cA^T(f(e))\simeq f^*(\cA^T)(e)$ which sends a subtree to its image via $f$. These maps are compatible with the operadic composition, as one has the equality 
$$ f^\cA_e\big(S^{\underline{e}}_e\circ_{\underline{e}} (R_1,\dots,R_n)\big)=S^{f(\underline{e})}_{f(e)}\circ_{f(\underline{e})} \big(f^\cA_{e_1}(R_1), \dots, f^\cA_{e_n}(R_n)\big)$$ for any choice of edges $\underline{e}=(e_1,\dots,e_n),e$ of $S$ for which $S(e_1,\dots,e_n;e)\neq \emptyset$. It follows that the family $f^\cA\coloneqq \{f^\cA_e\}_e$ assembles into a morphism of $S$-algebras $f^\cA\colon \cA^S\to f^*(\cA^T)$. Under the adjunction $(f_!,f^*)$, the transpose of $f^\cA$ yields a morphism of $T$-algebras  $$ f_*^\cA\colon f_!(\cA^S)\longrightarrow \cA^T.$$
\end{construction}

\noindent We are ready to define the functor whose left Kan extension will determine $\rho_!^P$.

\begin{deff}\label{rect1}
Let $P$ be a discrete operad. We define a functor $$ \rho^P\colon \omg/ P \longrightarrow \alg_P(\ssets)$$ as follows: for an object $(T,\alpha)$ in $\omg/ P$, we set $$ \rho^P(T,\alpha)\coloneqq \alpha_! (\cA^T).$$ 
Given another object $(S,\beta)$ and a morphism $h\colon (T,\alpha)\to (S,\beta)$ in $\omg/P$, the map $\rho^P(h)=h_*$ is defined via base change along  $\beta$, that is $$ h_* \coloneqq \beta_!(h^\cA_*)\colon\alpha_!(\cA^T)\longrightarrow \beta_!(\cA^S).$$ 
\end{deff}

\noindent Observe that any morphism of trees $f\colon S \to T$ can be regarded as a morphism $f\colon (S,f) \to (T,\id_T)$ in $\omg/ T$, in which case we have $f_*=f_*^\cA$.

\begin{deff}
Let $P$ be a discrete colored operad. The \emph{rectification functor} of $P$ is the functor $$\rho_!^{P}\colon \dsets/\cN_d P \to \alg_P(\ssets)$$ obtained as the left Kan extension of $\rho^P\colon \omg/P \to \alg_P(\ssets)$ along the Yoneda embedding $\omg/P\to \dsets/\cN_d P.$
\end{deff}

\noindent In particular, for a tree $T$, the simplicial $T$-algebra $\cA^T$ is the rectification of the identity morphism: $$ \rho_!^T(T,\id_T)=\cA^T.$$ 

\noindent More generally, we can describe the image of $\rho^P_!$ on a ge of $\dsets/\cN_d P$ more explicitly. Given a morphism $\alpha\colon T \to P$ and a object $c$ of $P$, let $\alpha/c$ be the poset whose objects are pairs $(\underline{e}, z),$ with $\underline{e}$ a tuple of edges of $T$ and $z$ an operation in $P(\alpha(\underline{e}); c)$, and where there is an arrow $$((\underline{e}^1,\dots,\underline{e}^n),z')\to (\underline{e},z) \ \text{ if and only if } \ z'=z\circ (\alpha(T^{\underline{e}^1}_{e_1}),\dots,\alpha(T^{\underline{e}^n}_{e_n})).$$ The simplicial set $\cN (\alpha/c)$ has the homotopy type of the union of the discrete multi-hom sets $P(\alpha(l_1),\dots,\alpha(l_m);c) $, where $(l_1,\dots,l_m)$ ranges over tuples of leaves of $T$.
The morphism $\alpha$ determines a representable $(T,\alpha)$ of $\dsets/\cN_d P$, and one has a canonical isomorphism $$\rho_!(T,\alpha)\simeq \cN(\alpha/c).$$

\noindent The adjoint functor theorem ensures the existence of an adjunction 
\begin{equation}\label{strun}
	\adjunction{\rho^P_!}{\dsets/\cN_d P}{\mathsf{Alg}_P(\ssets)}{\rho_P^*}.
\end{equation} 

\begin{deff}
	The  \emph{relative dendroidal nerve functor} of a discrete operad $P$ is the functor $\rho^*_P\colon \alg_P(\ssets)\to \dsets/\cN_d P$ right adjoint of the rectification functor.
\end{deff}

\begin{rmk}\label{recover}
	If $T$ is linear, that is, $T\simeq [n]\in \Delta$, for any object $i\in [n]$ there is an isomorphism $\cA^{[n]}(i)\simeq \Delta^i$, natural in $i$. In particular, given a discrete category $A$, the rectification functor $\rho^A_!$ of \Cref{rect1} coincides with Heuts-Moerdijk rectification functor $r_!^A\colon \ssets/\cN A \to \mathsf{Fun}(A,\ssets)$ defined in \cite[\S 4]{HM:LFHC}. By essential uniqueness of right adjoints, we can identify the functor $\rho^*_A$ with Lurie's relative nerve functor $\rho^*_A$ (\cite[\S 3.2.5]{Lu:HTT}). 
\end{rmk}

\noindent Let us give a more explicit description of the right adjoint.

\subsection{The relative dendroidal nerve functor}\label{sec1subsec2}

\noindent Consider a simplicial $P$-algebra $F$. The object $\rho^*_P(F)$ is a presheaf on $\omg/P$, henceforth to describe it it suffices to describe its values on the elements $(T,\alpha)$ in $\omg/P$. Fix one of such. By definition, one has $$\rho^*_P(F)_{(T,\alpha)}\coloneqq \Hom_{\dsets/\cN_dP}((T,\alpha), \rho^*_P(F))\simeq \Hom_{\alg_P}(\rho_!^P(T,\alpha), F).$$

\noindent Unwinding the definition, one has $$ \rho^*_P(F)_{(T,\alpha)}\simeq \mathsf{Hom}_{\alg_{P}(\ssets)} (\alpha_!(\cA^S),F) \simeq\mathsf{Hom}_{\alg_S(\ssets)}(\cA^S, \alpha^*F).$$ The elements in $\rho^*_P(F)_{(T,\alpha)}$ can be described as follows.

\begin{prop}\label{reduction1}
The following data are equivalent:
\begin{enumerate}
\item An element ${}^t\chi $ in $ (\rho^*F)_{(T,\alpha)} $.
\item A collection $$(\gamma_u \colon \Delta^k \to F({\alpha(e)}))_{u,e,k}$$ where $e$ ranges over the edges of $T$, $u$ over the non-degenerate $k$-simplices ${u\colon \Delta^k \to \cA^T(e)}$, $k$ ranges over natural numbers and the collection has to satisfy the following \emph{compatibility condition}: 
whenever $T(e_1,\dots,e_n;e)\neq \emptyset,$ for any face map $v\colon \Delta^k \to \Delta^{k'}$ fitting into a commutative square as the one on the left, the induced diagram on the right also commutes:
\begin{center}\begin{tikzcd}
\Delta^k \arrow[d, "v"'] \arrow[rr, "{(u_1,\dots,u_n)}"] &  & \cA^T({e_1})\times \dots \times \cA^T({e_n}) \arrow[d, "(T_e^{\underline{e}})_*"] &   & \Delta^k \arrow[d, "v"'] \arrow[rr, "{(\gamma_{u_1},\dots,\gamma_{u_n})}"] &  & F({\alpha(e_1)})\times \dots \times F({\alpha(e_n)}) \arrow[d, "(T_e^{\underline{e}})_*"] \\
 \Delta^{k'} \arrow[rr, "u"]                            &  & \cA^T(e)                                              &  & \Delta^{k'} \arrow[rr, "\gamma_u"]                                    &  & F({\alpha(e)})                                               
\end{tikzcd}
\end{center}
\end{enumerate}
\end{prop}

\begin{proof}
\noindent Suppose to have the data in $(1)$. By adjunction, the element ${}^t\chi$ corresponds to a map of $P$-algebras $\chi \colon \rho^P_!(T,\alpha)\to F$. Consider the morphism of $T$-algebras 
$$\eta\colon \cA^T\to \alpha^*\alpha_! \cA^T $$
given by the unit of the adjunction $(\alpha_!,\alpha^*)$ evaluated at $\cA^T$. It consists of maps of simplicial sets $\eta_e\colon \cA^T(e) \to \cN(\alpha/\alpha(e))$ for every edge $e$ of $T$ which are natural with respect to multimorphisms in $P$. Given an edge $e$ of $T$, a natural number $k\geq 0$ and a non-degenerate $k$-simplex $u\colon \Delta^k \to \cA^T(e)$, we define $\gamma_u$ as the composition $$\Delta^k\xrightarrow{u} \cA^T(e) \xrightarrow{\eta_e}\cN(\alpha/\alpha(e))\xrightarrow{\chi_{\alpha(e)}} F({\alpha(e)}).$$ It is straightforward to check that the compatibility condition holds, since $\eta$ and $\alpha^*(\chi)= \chi_{\alpha(-)}$ are maps of algebras.

\noindent Let us now prove that $(2)$ implies $(1)$, and suppose to be given a collection $\{\gamma_u\}_{u,e,k}$ as in point $(2)$. We construct the morphism of simplicial $P$-algebras $\chi\colon \rho_!(T,\alpha)\to F$ by constructing the maps of simplicial sets $\{\chi_c\colon \rho_!(\alpha)_c\to F_c\}_{c\in C(P)}$ by induction on the simplices, checking, at the $\ell$-th step, that we have  constructed a map of $P$-algebras in $\ell$-truncated simplicial sets, that is, that for any choice of objects $c_1,\dots,c_m;c$ and operation $w\in P(c_1,\dots,c_m;c)$, the following diagram in $\ssets_{\leq \ell}$ commutes:
\begin{equation}\label{mapalg}
\begin{tikzcd}
\cN(\alpha/c_1)_{\leq \ell}\times \dots \times \cN(\alpha/c_m)_{\leq \ell} \arrow[rr, "\chi_{c_1}\times \dots \times \chi_{c_m}"] \arrow[d, "w_*"'] &  & {F({c_1})}_{\leq \ell}\times\dots\times {F({c_m})}_{\leq \ell} \arrow[d, "w_*"] \\
\cN(\alpha/c)_{\leq \ell} \arrow[rr, "\chi_c"']                                                                                         &  & {F(c)}_{\leq \ell}                                              
\end{tikzcd}
\end{equation}

\noindent Let $\ell=0$, let $c$ be a object of $P$ and consider an element $(\underline{e},z)$ in $\cN(\alpha/c)_0$, with $\underline{e}=(e_1,\dots,e_n)$ and $z\in P(\alpha(e_1),\dots,\alpha(e_n);c)$. For every $i\in \{1,\dots,n\}$, let ${e_i}\colon \Delta^0\to \cA^T({e_i})$ be the morphism selecting $\eta=e_i$. Since $F$ is a $P$-algebra, we have a map $z_*\colon F({\alpha(e_1)})\times \dots \times F({\alpha(e_n)}) \to F(c)$, and set $$ \chi_c(\underline{e},z)= z_*\circ (\gamma_{{e_1}},\dots,\gamma_{{e_n}}) \colon \Delta^0\longrightarrow F(c).$$ This defines a map of sets $$(\chi_c)_0\colon \cN(\alpha/c)_0\to (F_c)_0.$$ To check commutativity of \Cref{mapalg}, consider $w\in P(c_1,\dots,c_m;c)$, and elements $z_i \in P(\alpha(e^i_1),\dots,\alpha(e^i_{k_i}); c_i)$, $i=1,\dots,m$; commutativity is proven by the following computation $$ w_*(\chi_{c_1}(\underline{e}^1,z_1),\dots,\chi_{c_m}(\underline{e}^m,z_m))= w_*\circ ((z_1)_*(\gamma_{{e_1^1}},\dots,\gamma_{{e^1_{k_1}}}), \dots, (z_m)_*(\gamma_{{e^m_1}},\dots, \gamma_{{e^m_{k_m}}}))= $$ 
$$= (w\circ (z_1,\dots,z_m))_*(\gamma_{{e_1^1}},\dots, \gamma_{{e_{k_m}^m}} )= \chi_c((\underline{e}^1,\dots,\underline{e}^m, w_*(z_1,\dots,z_m))),$$ where the last equality holds because of the compatibility condition.

\noindent Let $\ell=1$ and consider a non degenerate $1$-simplex $p\colon (\underline{e}',z')\to (\underline{e},z) \in \cN(\alpha/c)_1$, with $z'=z\circ (\alpha(T^{\underline{e}^1}_{e_1}),\dots,\alpha(T^{\underline{e}^n}_{e_n}))$ for subtrees $T^{\underline{e}^i}_{e_i} $, $i=1,\dots,n$, with $\underline{e}'=(\underline{e}^1,\dots,\underline{e}^n)$. The subtree inclusion $\eta=e_i\subseteq T^{\underline{e}^i}_{e_i}$ corresponds to a $1$-simplex $u_i\colon \Delta^1\to \cA^T({e_i})$, and we define $$\chi_c(p)= z_*\circ (\gamma_{u_1},\dots,\gamma_{u_m}),$$ where, if $u_i$ is degenerate we set $\gamma_{u_i}\coloneqq \gamma_{e_i}\sigma$, where $\sigma$ is the degeneracy. 

\noindent If $p$ is degenerate, that is $p=z\sigma$, we set  $\chi_c(p)\coloneqq \chi_c(\underline{e},z)\sigma$; compatibility with face maps is given by the compatibility condition with respect to the diagrams \begin{center}
\begin{tikzcd}
\Delta^0 \arrow[d, "\partial^1"'] \arrow[rr, "{({e_1^i},\dots, {e^i_{n_i}})}"] &  & \cA^T({e^i_1})\times \dots \times  \cA^T({e^i_{n_i}}) \arrow[d, "({T^{\underline{e}^i}_{e_i}})_*"] \\
\Delta^1 \arrow[rr, "u_i"]                                                         &  & \cA^T({e_i})                                                                                    \\
\Delta^0 \arrow[u, "\partial^0"] \arrow[rr, "e_i"]                             &  & \cA^T({e_i}) \arrow[u, "\id"']                                                                         
\end{tikzcd}
\end{center} for $i\in \{1,\dots,n\}$, and it still follows from the compatibility condition and the structure of $P$-algebra of $F$ that the diagram \Cref{mapalg} commutes also at the level of $1$-simplices. 

\noindent Let $\ell\geq 2$ and $t$ be a $\ell$-simplex of $\cN(\alpha/c)$, $t\colon (\underline{e}^{(l)},z^{(\ell)})\to \dots \to (\underline{e}^{(1)},z^{(1)})\to (\underline{e}^{(0)},z^{(0)})$. The $\ell$-simplex $\chi_c(t)$ of $F(c)$ is defined as follows. For any $i\in \{1,\dots,n\}$, let $u_i\colon \Delta^\ell\to \cA^T({e_i})$ be a $\ell$-simplex corresponding to the chain of subtrees with root $e_i$ whose subsequent composition yields $z^{(i)}$. 
If each of the $u_i$'s is non-degenerate, then by hypothesis we have maps $\gamma_{u_i}$ and we set 
$$\chi_c(t)\coloneqq z_* (\chi_{u_1},\dots,\chi_{u_n}).$$ If some $u_i$ is degenerate, we define $\gamma_{u_i}\coloneqq\gamma_{\overline{u}_i}\circ\sigma$ for $\overline{u}_i$ the unique non-degenerate simplex and $\sigma$ a degeneracy such that $u_i=\overline{u}_i\circ \sigma$, and then define $\chi_c(t)$ in the same way.

\noindent It is a lengthy but straightforward computation to check that the simplicial relations involving face maps are satisfied thanks to the compatibility condition enjoyed by the family $\{\gamma_u\}_{e,u,k}$ and that those involving degeneracies are satisfied by construction, and just as readily, one checks that the diagram in \Cref{mapalg} commutes at the level of $\ell$-simplices.

\noindent We have hence defined a map of simplicial $P$-algebras $\chi\colon \rho_!(\alpha)\to F$. By adjunction, it yields an element in $\rho^*F_{(T,\alpha)}$, and this concludes the proof.
\end{proof}

\noindent When the domain of $\alpha$ is a linear tree, i.e. an element of $\Delta$, the above description can be further simplified.

\begin{coro}\label{panevino}
Cosnider a representable of $\dsets/\cN_d P$ of the form $([n], \alpha)$, $n\geq 0$, and let $F$ be a simplicial $P$-algebra. An element ${}^t\chi\in \rho^*(F)_{([n],\alpha)}$ is determined by a family of morphisms $$\{\gamma_u\colon \Delta^k \to F({\alpha(u(k))})\}_u$$ as  $u$ ranges over the maps $u\colon \Delta^k\to \Delta^n$, which satisfies the following property: given maps $u,u'$ and a face map $v\colon \Delta^k\to \Delta^{k'}$ making the diagram below on the left commute, the induced diagram on the right commutes as well,
\begin{center}
\begin{tikzcd}
\Delta^k \arrow[d, "v"'] \arrow[rd, "u"] &          &  & \Delta^k \arrow[d, "v"'] \arrow[r, "\gamma_u"] & F({\alpha(u(k)}) \arrow[d, "f_v"] \\
\Delta^{k'} \arrow[r, "u'"']             & \Delta^n &  & \Delta^{k'} \arrow[r, "\gamma_{u'}"]        & F({\alpha(u'(k'))})              
\end{tikzcd}
\end{center} 
where $f_v$ denotes the composition of the morphisms from $\alpha(u(k))$ to $\alpha(u'(k'))$.
\end{coro}

\begin{proof} This follows from \Cref{reduction1} and the following considerations. First, as observed in \Cref{recover}, there is a natural isomorphism $\cA^{[n]}(j)\simeq \Delta^j$ for any $j\in [n]$. Second, any map of simplicial sets $w\colon \Delta^k \to \Delta^n$ uniquely factors through the inclusion $i\colon \Delta^{w(k)}\hookrightarrow \Delta^n$, $i(x)=x$, and writes as the composition $\Delta^k \xrightarrow{\overline{w}}\Delta^{w(k)}\simeq \cA^{[n]}({w(k)})\xhookrightarrow{i}\Delta^n$. 

\noindent Given a face map $v\colon \Delta^k\to \Delta^{k'}$, the commutative triangle $u=u'\circ v$ corresponds to the commutative square $ z\circ \overline{w}= \overline{w'}\circ v,$ where $z\colon \cA^{[n]}({w(k)})\to \cA^{[n]}({w'(k')})$ is the image of the arrow $w(k')\leq w'(k')$ via the functor $\cA^{[n]}\colon [n]\to \ssets$, and commutativity of the square on the right is just the compatibility condition of \Cref{reduction1} applied to $ z\circ \overline{w}= \overline{w'}\circ v,$.

\noindent Vice versa, if we have a map $u\colon \Delta^n \to \cA^{\Delta^n}(j)\simeq \Delta^j$, then we produce a morphism $\tilde{u}\colon \Delta^n \to \Delta^k$ by postcomposing with the same natural inclusion $\Delta^j \hookrightarrow \Delta^k $.
\end{proof}

\noindent We deduce the immediate simplification of the above criterion.

\begin{coro}\label{evenmore}
In the situation of \Cref{panevino}, an $n$-simplex $\chi\in (\rho^*F)_{\alpha}$ is completely determined by a sequence of simplices $\chi_i\colon \Delta^i \to F_{\alpha(i)}$, for $i\in \{0,1,\dots,n\}$, such that $f_i (\chi_{i-1})=d_i \chi_i$.
\end{coro}

\noindent We conclude with the following important property of the relative dendroidal nerve functor.

\begin{lemma}\label{freegaza}
Let $F$ be a simplicial $P$-algebra, $c$ a object of $P$, and denote by $(\rho^*F)_c$ the fibre of $\rho^*F$ over $c$. There is an isomorphism of simplicial sets $$(\rho^*F)_c \simeq F(c)$$ natural in $P$-algebras, meaning that, given a map of $P$-algebras $\varphi\colon F \to G$, the induced map between the fibres $\rho^*\phi_c\colon \rho^*F_c \to \rho^*G_c$ is the map $\varphi_c\colon F(c) \to G(c)$.
\end{lemma}
 
\begin{proof}

\noindent Given a morphism of dendroidal sets $E\to X$ and an object $c$ of $X$, the fibre of $E\to X$ over $c$ is the simplicial set $E_c$ obtained as the pullback
\begin{center}
	\begin{tikzcd}
		E_c \arrow[r]\arrow[d] & E \arrow[d] \\
		\eta \arrow[r, "\{c\}"] & X
	\end{tikzcd}
\end{center}
\noindent where $\{c\}\colon \eta \to X$ is the map selecting $c$ as object. If $E\to X$ is a dendroidal left fibration, the map $E_c\to \eta$ is a left fibration of simplicial sets as well; as $\eta$ is a Kan complex, the map $E_c	\to \eta$ is also a right fibration, which means that $E_c$ is a Kan complex.

\noindent By definition, for any element $(E,p)$ in $\dsets/\cN_d P$, there is a canonical isomorphism of simplicial sets $$E_c\simeq \hom_{X}(\{c\}, E),$$ where  $\{c\}\colon \eta=\Delta^0 \to \cN_d P$ is the map selecting $c$.  In particular, the fibre of $\rho^*F$ over $c$ can be described as the pullback $$\rho^*F_c= \hom_{\cN_d P}(\{c\}, \rho^*F);$$ an $n$-simplex of $\hom_{\cN_d P}(\{c\}, \rho^*F)$ is a  morphism $\Delta^n\to \rho^*F$ over $\cN_d P$, where the map ${\Delta^n \to \cN_d P}$ is given by the composition $\Delta^n \to \Delta^0=\eta \xrightarrow{c} \cN_d P$. By \Cref{panevino} and \Cref{evenmore}, any such simplex is determined by a single element $x_n \colon \Delta^n \to F(c)$, and one can check that the bijection is natural in $\Delta$. This concludes the proof.
\end{proof}

\section{The Quillen adjunction}\label{section5}

\noindent Fix a discrete operad $P$ and write $\rho_!$, resp. $\rho^*$, for the functor $\rho_!^P$, resp. $\rho^*_P$. The main goal of this section is to prove that the pair $(\rho_!, \rho^*)$ is a Quillen adjunction with respect to the covariant model structure for dendroidal left fibrations and the projective model structure on simplicial $P$-algebras. We will prove this in \Cref{adjj}, and our strategy consists in showing that $\rho^*$ is a right Quillen functor. To this end, we need to better understand lifts of maps of the form $\rho^*(F)\to \rho^*(G)$ against dendroidal boundary and leaf horn inclusions. Let us first begin with some combinatorial construction.
\subsection{Root preserving faces and non-degenerate maximal chains for a tree}\label{matija} Let us start by giving some definitions.
\begin{deff}
	We say that a face map $\partial \colon S \to T$ is \emph{root preserving} if $\partial(r_S)=r_T$; for a face of linear trees $\partial\colon [i]\to [n]$, this means that $\partial(i)=n$.
\end{deff}

\noindent Let $T$ be a tree. A $n$-simplex $u$ of $\cA^T(r_T)$ can be written as a sequence of reversed subtree inclusions of the form $$u \colon \ \ T_0\to \dots \to T_{n-1} \to T_n,$$ and $u$ is non-degenerate if and only if in $T_i$ there is at least one more vertex than $T_{i+1}$.
  
\begin{deff} Let $T$ be a tree. A \emph{non-degenerate maximal chain} for $T$ is a simplex $$ u\colon \Delta^n \to \cA^T({r_T}), \quad u \colon \ \ T_0\to \dots \to T_{n-1} \to T_n $$ such that $T_0=T$, $T_n=\eta=r_T$ and each subtree $T_{i}$ is obtained from $T_{i+1}$ by adding exactly one vertex (and all its input edges). 
\end{deff}

\noindent Put differently, a non-degenerate maximal chain for $T$ is sequence of root preserving subtree inclusions $T_n\subseteq \dots \subseteq T_{1} \subseteq T_{0}$ in $\omg$ with $T_n=r_T$, $T_0=T$ and where each $T_i$ is obtained from $T_{i+1}$ by adding exactly one vertex. In particular, observe that in general a non-degenerate maximal chain may not be unique, but the number $n$ is, as $n= \# V(T)$.

\begin{example}\label{ex1}
For a tree $T$ of the form

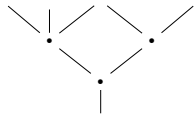
\adjustbox{scale=0.5, center}{
\begin{tikzcd}
	{} \arrow[rd, no head] & {}                                                                 & {} \arrow[rd, no head]                         &                             & {} \\
	& \bullet \arrow[ru, no head] \arrow[rd, no head] \arrow[u, no head] &                                                & \bullet \arrow[ru, no head] &    \\
	&                                                                    & \bullet \arrow[ru, no head] \arrow[d, no head] &                             &    \\
	&                                                                    & {}                                             &                             &   
	\end{tikzcd}}
there are two non-degenerate maximal chains, given by

\adjustbox{scale=0.8, center}{
\begin{tikzcd}
	{} & {} \arrow[d, no head]                                               & {} \arrow[rd, no head]                         &                             & {}                     & {} \arrow[rd, no head]                         &                             & {}                     &                                                &              &    &                       \\
	& \bullet \arrow[ru, no head] \arrow[rd, no head] \arrow[lu, no head] &                                                & \bullet \arrow[ru, no head] & {} \arrow[rd, no head] &                                                & \bullet \arrow[ru, no head] & {} \arrow[rd, no head] &                                                & {}           &    &                       \\
	&                                                                     & \bullet \arrow[ru, no head] \arrow[d, no head] & {} \arrow[r]                & {}                     & \bullet \arrow[ru, no head] \arrow[d, no head] & {} \arrow[r]                & {}                     & \bullet \arrow[ru, no head] \arrow[d, no head] & {} \arrow[r] & {} & {} \arrow[d, no head] \\
	&                                                                     & {}                                             &                             &                        & {}                                             &                             &                        & {}                                             &              &    & {}                   
	\end{tikzcd}}
\noindent and
\adjustbox{scale=0.8, center}{
\begin{tikzcd}
	{} & {} \arrow[d, no head]                           & {} \arrow[rd, no head]                                             &                             & {} \arrow[rd, no head] & {} \arrow[d, no head]                           & {}                                             &              &                        &                                                &              &    &                       \\
	& \bullet \arrow[ru, no head] \arrow[lu, no head] &                                                                    & \bullet \arrow[ru, no head] &                        & \bullet \arrow[rd, no head] \arrow[ru, no head] &                                                & {}           & {} \arrow[rd, no head] &                                                & {}           &    &                       \\
	&                                                 & \bullet \arrow[lu, no head] \arrow[ru, no head] \arrow[d, no head] & {} \arrow[r]                & {}                     &                                                 & \bullet \arrow[ru, no head] \arrow[d, no head] & {} \arrow[r] & {}                     & \bullet \arrow[ru, no head] \arrow[d, no head] & {} \arrow[r] & {} & {} \arrow[d, no head] \\
	&                                                 & {}                                                                 &                             &                        &                                                 & {}                                             &              &                        & {}                                             &              &    & {}                   
	\end{tikzcd}}

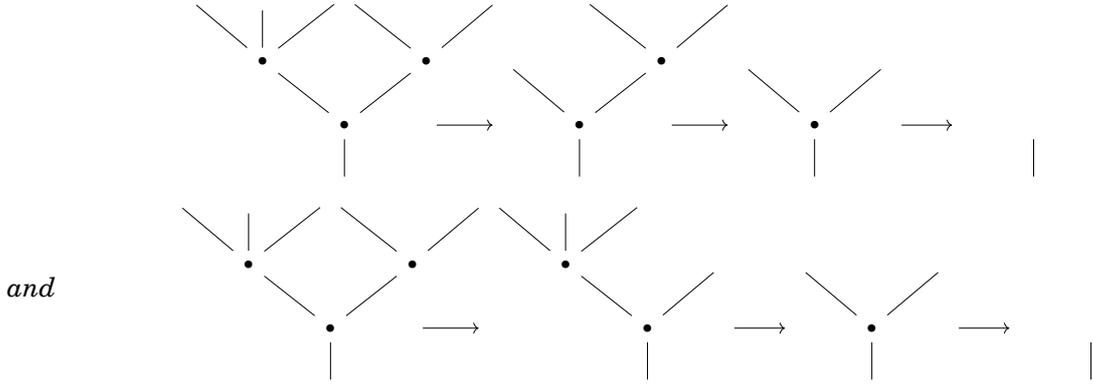
\captionof{figure}{The two non-degenerate maximal chains for $T$}\label{figure2}
\end{example}

\begin{rmk} When defining the partial order on $P(T^\uparrow_e)$ (\Cref{sirene}), we could have chosen the opposite order, that is, the natural inclusion of subtrees, which would make the definition above more natural. However, by choosing \emph{reverse} inclusion we obtain a correspondence between dendroidal left fibrations and \emph{left}, and not \emph{right}, fibrations of simplicial sets, see \Cref{projfibleftfib}.
\end{rmk}

\noindent Consider a root preserving face map $\partial\colon S \to T$. We still denote by $\partial$ the map of simplicial sets $$ \partial\colon \cA^S(r_S)\longrightarrow \cA^T(r_T)$$ induced by the map of posets sending a subtree in $P(S_{r_S}^\uparrow)\to P(T_{r_T}^\uparrow)$ which sends a subtree in $S$ to its image in $T$ via $\partial$. 

\noindent The next result relates root preserving face maps and non-degenerate maximal chains of trees. It will be fundamental in the proofs of \Cref{reducelift} and \Cref{projfibleftfib}, as well as in the previous constructions \Cref{construction1} and \Cref{construction2}.

\begin{lemma}\label{lemmaaux}
Let $\overline{u}\colon \Delta^n \to \cA^T({r_T})$ be a non-degenerate maximal chain and $d\colon \Delta^i \to \Delta^n$ a root preserving face. 

\begin{enumerate}[label=(\roman*)]
	\item There exist a tree $ T(d,\overline{u})$, a root preserving dendroidal face $\partial\colon T(d,\overline{u}) \to T$ and a non-degenerate simplex $u\colon \Delta^i\to \cA^{T(d,\overline{u})}(r_{T(d,\overline{u})})$ such that the following diagram commutes
	
	\begin{equation}\label{febbraio}
		\begin{tikzcd}
			\Delta^i \arrow[r, "d"] \arrow[d, "u"']                                 & \Delta^n \arrow[d, "\overline{u}"] \\
			{\cA^{T(d,\overline{u})}({r_{T(d,\overline{u})}}}) \arrow[r, "\partial"] & \cA^T({r_T})                       
		\end{tikzcd}
	\end{equation} and such that the triple $(T(d,\overline{u}), \partial,u)$ is initial amongst the triples satisfying the properties of above. In particular, such $(T(d,\overline{u}), \partial,u)$ is essentially unique.
	\item If $d(i-1)<n-1$ or $i=1<n$, the face map $\partial\colon T(d,\overline{u})\to T$ is not an isomorphism of trees. In other words, the tree $T(d,\overline{u})$ and the face map $\partial$ contribute to the colimit defining the dendroidal set $\partial T$.

\end{enumerate}

\end{lemma}

\begin{proof} Let us first prove point (i). We represent the non-degenerate maximal chain $\overline{u}$ as a chain of root preserving subtree inclusions $\overline{u}_n \subseteq \overline{u}_{n-1}\subseteq\dots \subseteq \overline{u}_0$, with $\overline{u}_n=r_T$, $\overline{u}_0=T$ and such that at every step we add exactly one vertex of the tree. 

\noindent Observe that the diagram in \Cref{febbraio} is in fact the nerve of the diagram of categories, and more precisely of posets, given by the following
\begin{equation}\label{the}
	\begin{tikzcd}
{[i]}\arrow[r, "d"] \arrow[d, "u"']                                 & {[n]} \arrow[d, "\overline{u}"] \\
		P(T(d,\overline{u})^\uparrow_{r_{T(d,\overline{u})}}) \arrow[r, "\partial"] & P(T^\uparrow_{r_T}) \ . 
	\end{tikzcd}
\end{equation}

\noindent We construct the triple $(T(d,\overline{u}), \partial, u)$ inductively. More precisely, we start by inductively constructing a chain of root-preserving inclusions of trees $u_i\subseteq u_{i-1}\subseteq \dots \subseteq u_0$, the base case being $k=i$, and dendroidal face maps ${\partial_k\colon u_k\to \overline{u}_{d(k)}}$ for any $k\in [i]$, such that:
\begin{enumerate}
	\item for every $k\in [i]$, the face $\partial_k$ is root preserving and $\partial(u_k) = \overline{u}_{d(k)}$, that is, the diagram in \Cref{the} commutes at the level of objects,
	\item the square
	\begin{equation}\label{commutdiag}
		\begin{tikzcd}
			u_k \arrow[r, hook] \arrow[d, "\partial_k"'] & u_{k-1} \arrow[d, "\partial_{k-1}"] \\
			\overline{u}_{d(k)} \arrow[r, hook]         & \overline{u}_{d(k-1)}              
		\end{tikzcd}
	\end{equation}
	commutes for any $k>0$, that is, the diagram in \Cref{the} commutes at the level of morphisms,
	\item at the $k^{\text{th}}$ step, the triple $(u_k, \overline{\partial}_k, u_k\to u_{k-1}\to \dots \to u_{i-1}\to u_i )$, where $\overline{\partial}_k$ is the dendroidal face given by the composition $ u_k\xrightarrow{\partial_k} \overline{u}_{d(k)} \subseteq T(d,\overline{u})$, is initial with respect to the tree $T$, the non-degenerate maximal chain $\overline{u}$ and the root preserving face map given by the composition $\Delta^{i-k}\hookrightarrow \Delta^i \xrightarrow{d} \Delta^n$, where the inclusion $\Delta^{i-k}\hookrightarrow \Delta^i$ is the face map defined by the assignment $[i-k]\ni j\mapsto j+k\in [i]$.
\end{enumerate} 
\noindent Observe that condition $(1)$ implies that $\partial_k$ is an inner face map, so in particular it also induces a bijection between the sets of leaves of the tree in the domain and in the codomain.

\noindent Let $k=i$. As $d$ is root preserving, we have $\overline{u}_{d(i)}=\overline{u}_n=\eta$, and we define  $u_i\coloneqq\eta$ and $\partial_i=\id \colon u_i\to \overline{u}_n$. This is clearly an initial inner face map (condition $(2)$ at this stage is empty). 

\noindent Consider $k=i-1$. As $\overline{u}$ is non-degenerate, the tree $\overline{u}_{d(i-1)}$ has at least one vertex. Let $\ell$ be the number of leaves of $\overline{u}_{d(i-1)}$ and $C_\ell$ the corolla with $\ell$ leaves. There exists an essentially unique inner face map from $ C_\ell \to \overline{u}_{d(i-1)}$, which is the identity if $d(i-1)=n-1$, and which is clearly initial amongst the inner face maps $S\to \overline{u}_{d(i-1)}$. We set $u_{i-1}\coloneqq C_\ell$, that is, the corolla obtained by contracting all the inner edges of $\overline{u}_{d(i-1)}$, and we define $\partial_{i-1}\colon u_{i-1}\to \overline{u}_{d(i-1)}$ as the unique inner face map just mentioned. By construction, the conditions $(1), (2)$ and $(3)$ are satisfied.

\noindent Let $0< k\leq i-1$, and suppose we have constructed a chain of root preserving inclusions $u_i\subseteq\dots \subseteq u_k$ satisfying the requirements; we construct $(u_{k-1},\partial_{k-1})$ as follows. Write $ \overline{u}_{d(k-1)}$ as the grafting $$ \overline{u}_{d(k-1)}= \overline{u}_{d(k)}\circ (\overline{v}_1,\dots,\overline{v}_{\ell_k}),$$ for some uniquely determined subtrees $\overline{v}_j$. Since $\overline{u}$ is non-degenerate, at least one $\overline{v}_j$ is different from $\eta$. For any $j=1,\dots,m_k$, if $\overline{v}_j\neq \eta$, we denote by $v_j$ the corolla with the same number of leaves as $\overline{v}_j$ and write $\chi_j$ for the unique inner face map $v_j \to \overline{v}_j$; we set $v_j=\eta$ and $\chi_j=\id$ otherwise.
We define  $$u_{k-1}\coloneqq u_k\circ (v_1,\dots,v_{m_k}), $$ and the face map $\partial_{k-1}$ as the grafting $$ \partial_{k-1}\coloneqq \partial_k \circ (\chi_1,\dots,\chi_{m_k}).$$ By construction, $\partial_{k-1}(u_{k-1})=\overline{u}_{d(k-1)}$ and the diagram in \Cref{the} commutes. Moreover, as each of the $\chi_j$ is initial and by inductive hypothesis $\partial_k$ is initial as well, we have that $\partial_{k-1}$ initial as well, as wanted.

\noindent In conclusion, we have shown that the triple given by $(T(d,\overline{u}), \partial, u)$ is initial, where $T(d,\overline{u})=u_0$, $\partial=\overline{\partial}_0\colon T(d,\overline{u})=u_0 \xrightarrow{\partial_0} \overline{u}_{d(0)} \subseteq \overline{u}_n = T $ and $u$ is the non-degenerate simplex defined by $u_0\to \dots u_{i-1}\to u_i$. It is straightforward to check that initiality implies essential uniqueness, so this concludes the proof of (i). 

\noindent Let us now address point (ii).  If $d(i-1)<n-1$, maximality of $\overline{u}$ ensures that in $\overline{u}_{d(i-1)}$ there are at least  two vertices, which is equivalent to say that $\overline{u}_{d(i-1)}$ has at least one inner edge. By construction, $\partial_{i-1}\colon u_{i-1}\to \overline{u}_{d(i-1)}$ is the face obtained by contracting all the inner edges of $\overline{u}_{d(i-1)}$, it is not an isomorphism. As $\partial_{i-1}=\partial_{|_{u_{i-1}}}$, the dendroidal face $\partial$ cannot be an isomorphism of trees either, as claimed. On the other hand, when $i=1$ the tree $T(d,\overline{u})=u_{0}$ is always a corolla, but if $n>1$, then $T$ has at least two vertices, which means that the face map $\partial\colon T(d,\overline{u})\to T$ cannot be an isomorphism of trees.
\end{proof}

\begin{example}
	Consider the tree $T$ of \Cref{ex1} and let $\overline{u}$ be the first non-degenerate maximal chain of \Cref{figure2}. Let $d\colon \Delta^2\to\Delta^3$ be the face map defined by $d(0)=0, d(1)=1,d(2)=3$. Then we have that $(T(d,\overline{u}),\partial, u)$ is determined by the non-degenerate simplex $u$ (in this case also maximal) given by 
	
	\adjustbox{scale=0.6,center}
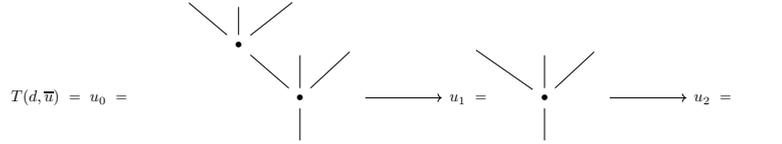
{
		\begin{tikzcd}
			& {} \arrow[rd, no head] & {} \arrow[d, no head]       & {} \arrow[ld, no head]                         &               &  &                        &                            &                        &  &         &                       \\
			&                        & \bullet \arrow[rd, no head] & {} \arrow[d, no head]                          & {}            &  & {} \arrow[rd, no head] & {} \arrow[d, no head]      & {} \arrow[ld, no head] &  &         &                       \\
			{T(d,\overline{u})\ = \ u_0 \ =} &                        &                             & \bullet \arrow[ru, no head] \arrow[d, no head] & {} \arrow[rr] &  & u_1 \ =                & \bullet \arrow[d, no head] & {} \arrow[rr]          &  & u_2 \ = & {} \arrow[d, no head] \\
			&                        &                             & {}                                             &               &  &                        & {}                         &                        &  &         & {}                   
	\end{tikzcd}}
 where the face map $\partial $ is given by the inner face map 
 
 \adjustbox{scale=0.6,center}
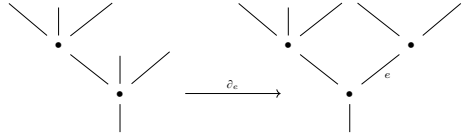
{
 	\begin{tikzcd}
 		{} \arrow[rd, no head] & {} \arrow[d, no head]       & {} \arrow[ld, no head]                         &                             & {} \arrow[rd, no head] & {} \arrow[d, no head]                           & {} \arrow[rd, no head]     &                                                      & {} \\
 		& \bullet \arrow[rd, no head] & {} \arrow[d, no head]                          & {}                          &                        & \bullet \arrow[rd, no head] \arrow[ru, no head] &                            & \bullet \arrow[ld, "e", no head] \arrow[ru, no head] &    \\
 		&                             & \bullet \arrow[ru, no head] \arrow[d, no head] & {} \arrow[rr, "\partial_e"] &                        & {}                                              & \bullet \arrow[d, no head] &                                                      &    \\
 		&                             & {}                                             &                             &                        &                                                 & {}                         &                                                      &   
 \end{tikzcd}}
 \noindent obtained by contracting the inner edge $e$.
 
 \noindent Consider now the face map $d$ defined by $d(0)=0, d(1)=2, d(2)=3$, which does not fall in the hypothesis of point (ii) of \Cref{lemmaaux}. In this case we have $u_2=\eta$, $u_1=C_2$ and $u_2=T$ and $\partial=\id$.
\end{example}

\noindent We conclude this part with the following 
\begin{deff}\label{maxex}
Given a tree $T$ and a non-degenerate simplex $w\colon \Delta^k \to \cA^T({r_T})$, a \emph{maximal extension} of $w$ is a non-degenerate maximal chain $\overline{w}\colon \Delta^n \to \cA^T({r_T})$ together with a face map $d\colon \Delta^k\to \Delta^n$ such that $w=\overline{w}\circ d$ .
\end{deff}

\noindent Observe that, even if there are multiple maximal extensions of $w$, the face map $d\colon \Delta^k\to \Delta^n$ is unique. 

\subsection{Lifts against dendroidal boundaries}\label{lift1} Given a tree $T$, we write $\iota\colon \partial T \to T $ for the dendroidal boundary inclusion. If $T$ is the domain of a representable $(T,\alpha)$ in $\dsets/\cN_d P$, one obtains another element (no longer representable) $(\partial T, \alpha\iota)$ by precomposing with $\iota$. 

\noindent Given a simplicial $P$-algebra $F$, a tree $T$ and a non-degenerate maximal chain ${\overline{u}\colon \Delta^n \to \cA^T({r_T})}$ of $T$, in \Cref{construction1} we explain how to associate to a morphism $\chi\colon (\partial T, \alpha \iota)\to \rho^*(F)$ a map of simplicial sets $\chi_{\overline{u}}\colon \partial \Delta^n \to F({\alpha(r_T)})$. Recall that $\partial T$, resp. $\partial \Delta^n$, is defined as the colimit of the face maps into $T$, resp. $\Delta^n$.

\begin{construction}\label{construction1} Let $F, \overline{u}$ and $\chi$ be as above. To construct a morphism $(\partial\chi)_{\overline{u}}\colon \partial \Delta^n \to F({\alpha(r_T)})$ we define, for any face $d\colon \Delta^k\to \Delta^n$ and coherently on nested faces, a map of simplicial sets  $X_{\overline{u},d}\colon \Delta^k \to F({\alpha(r_T)})$.

\noindent Let us use the characterization of the dendrices of $\rho^*(F)$ detailed in \Cref{reduction1}. 
\noindent Suppose first that $d$ is a face satisfying the hypotheses of point (ii) of \Cref{lemmaaux}, and let $(T(d,\overline{u}), \partial,u)$ be the triple given by the lemma. The restriction of $\alpha\iota$ to the face $\partial\colon T(d,\overline{u})\to T$ yields a map $\chi_{u}\colon \Delta^k\to F({\alpha(\partial(r_{T(d,\overline{u})})})=F({\alpha(r_T)}) $, and we set $X_{\overline{u},d}\coloneqq \chi_u$.

\noindent If $d$ is root preserving and $d(k-1)=n-1$, let $\{e_1,\dots,e_m\}$ be the input edges of the root corolla of $T$, that is, the leaves of the tree $\overline{u}_{n-1}$. For every $i=1,\dots,m$, there exist unique chains $v^i\colon \Delta^k\to \cA^T(e_i)$ for which we can write $\overline{u}\circ d$ as the composition 
$$\Delta^k \xlongrightarrow{(v^1,\dots,v^m)} \cA^T(e_1)\times \dots \cA^T(e_m)\xlongrightarrow{(\overline{u}_{n-1})_*} \cA^T(r_T).$$ As $\cA^T(e_i)\simeq \cA^{T^\uparrow_{e_i}}(r_{T^\uparrow_{e_i}})$ and the external face map $T^\uparrow_{e_i}\to T$ is not the identity, there exists maps $\chi_{v^i}\colon \Delta^k\to F(\alpha(e_i))$ for every $i=1,\dots,m$, and we define $X_{\overline{u},d}$ as the composition $$ \Delta^k\xlongrightarrow{\chi_{v^1},\dots,\chi_{v^m}} F(\alpha(e_1))\times \dots \times F(\alpha(e_m))\xlongrightarrow{(\overline{u}_{n-1})_*} F(\alpha(r_T)).$$

\noindent If $d$ is not root preserving, we proceed in a similar way. First, we observe that by maximality of $\overline{u}$, the subtree $\overline{u}_{d(k)}$ has at least one vertex, and we denote by $\{e_1,\dots,e_m\}$ its leaves $\overline{u}_{d(k)}$. For all $i=1,\dots,m$ there exists a unique chain $v^i\colon \colon \Delta^k\to \cA^T({e_i})$ with the property that $\overline{u}\circ  d=(v^1,\dots, v^m)\circ (\overline{u}_{d(k)})_*$, and we define $X_{\overline{u},d}$ as the composition $X_{\overline{u},d}\coloneqq  (\overline{u}_{d(k)})_* \circ (\chi_{\overline{u}_1},\dots,\chi_{\overline{u}_m})$.

\noindent One readily checks coherence of nested simplicial face, so we have a map ${(\partial\chi)_{\overline{u}}\colon \partial\Delta^n\to F({\alpha(r_T)}),}$ as wanted.  
\end{construction}

\noindent We use this construction for the following

 \begin{prop}\label{reducelift}
Let $f\colon F \to G$ be a morphism of $P$-algebras and $(T,\alpha)$ a representable in $\dsets/\cN_d P$. To have a dotted lift in the solid diagram

\[
\begin{tikzcd}
(\partial T, \alpha\circ \iota) \arrow[r, "\chi"] \arrow[d] & \rho^*(F) \arrow[d] \\
(T,\alpha) \arrow[r, "\xi"'] \arrow[ru, dashed, "\Lambda"]& \rho^*(G) 
\end{tikzcd}
\]
it suffices to have a lift in the diagram of simplicial sets
\begin{center}
\begin{tikzcd}
\partial\Delta^n \arrow[r, "(\partial \chi)_u"] \arrow[d]               & F({\alpha(r_T)}) \arrow[d, "f_{\alpha(r_T)}"] \\
\Delta^n \arrow[r, "\xi_u"'] \arrow[ru, "\lambda_u", dashed] & G({\alpha(r_T)})                           
\end{tikzcd} 
\end{center}
\noindent for any non-degenerate maximal chain $u\colon \Delta^n \to \cA^T({r_T})$. 
\end{prop}

\begin{proof}
Let $(\lambda_u)_{u}$ be a family of lifts as in the statement. Using the criterion of \Cref{reduction1} we construct, for any non-degenerate ${w\colon \Delta^k \to \cA^T(e)}$, a map of simplicial sets ${\Lambda_w\colon \Delta^k\to F({\alpha(e)})}$ such that:
\begin{enumerate}
\item the family $\{\Lambda_w\}_w$ satisfies the compatibility condition, defining a map $\Lambda\colon (T,\alpha) \to \rho^*F$ in $\dsets/\cN_d P$;
\item for any face $\partial\colon S\to T$ such that $w$ factors as $$\Delta^k\xrightarrow{w'} \cA^S({e'})\to \cA^T({\partial(e'))}=\cA^T(e)$$ for some $w'$, one has $\Lambda_w=\chi_{w'}$. This means that  $ \Lambda \iota = \chi$;
\item for any edge $e$ of $T$ and $w$ non-degenerate, $f_{\alpha(e)}\Lambda_w= \xi_w$, which means that $\rho^* (f) \Lambda=\xi$.
\end{enumerate} 

\noindent Fix then an edge $e$ of $T$ and a non-degenerate simplex $w\colon \Delta^k\to \cA^T(e)$. 

\noindent If $e\neq r_T$, the tree inclusion $\partial\colon T_e^{\uparrow}\hookrightarrow T$ yields the equality $\cA^T(e)=\cA^{T_e^{\uparrow}}({r_{T_e^{\uparrow}}})$. The restriction of $\chi$ to $\partial$ yields a map $\chi_{|_{\partial}}\colon (T_e^{\uparrow},\alpha\partial) \to \rho^*(F)$. By \Cref{panevino}, this gives a simplex ${\chi_w\colon \Delta^k \to F({\alpha(\partial(r_S))})=F({\alpha(e)}),}$ and we define $\Lambda_w\coloneqq \chi_w$. 

\noindent If $e=r_T$ and $w$ is maximal, we set $\Lambda_w\coloneqq \lambda_w$, which exists by assumption. If $w$ is not maximal, choose a maximal extension $(\overline{w},d)$ of $w$, with $d\colon \Delta^k\to \Delta^n$ the unique face such that $w=\overline{w}\circ d$, and define $\Lambda_w\coloneqq\lambda_{\overline{w}}\circ d= {(\partial\chi)_{\overline{w}}}_{|_{d}}$. Observe that, if $d$ satifies the hypotheses of point (ii) of \Cref{lemmaaux}, the map ${(\partial\chi)_{\overline{w}}}_{|_{d}}$ corresponds to $\chi_{w'}$, where $w'\colon \Delta^k\to \cA^{T(d,\overline{w})}({r_{T(d,\overline{w})}})$ is the simplex such that $\overline{w}\circ d= w=\partial\circ w'$, with $\partial \colon T(d,\overline{w})\to T$ the face map given by the lemma.

\noindent Let us now prove that the collection of the $\{\Lambda_w\}_w$ just defined satisfies the three conditions above.

\begin{enumerate}
\item Consider edges $e_1,\dots,e_n,e$ of $T$ such that $ T(e_1,\dots,e_n; e)\neq \emptyset$, a face map $v\colon \Delta^k\to \Delta^{k'}$ and a commutative diagram
\begin{equation}\label{prima}\begin{tikzcd}
\Delta^k \arrow[rr, "{(u_1,\dots,u_n)}"'] \arrow[d, "v"'] & & \cA^T({e_1})\times \dots \times \cA^T({e_n}) \arrow[d] \\
\Delta^{k'} \arrow[rr, "u"]                               & &\cA^T(e)                                             
\end{tikzcd}\end{equation} 
where $u_1,\dots,u_n,u$ are non-degenerate. We need to check that the following diagram commutes:
\begin{equation}\label{commut}\begin{tikzcd}
\Delta^k \arrow[rr, "{(\Lambda_{u_1},\dots,\Lambda_{u_n})}"'] \arrow[d, "v"'] && F({\alpha(e_1)})\times \dots \times F({\alpha(e_n)}) \arrow[d] \\
\Delta^{k'} \arrow[rr, "\Lambda_u"]                               & &F({\alpha(e)})                                             
\end{tikzcd}\end{equation} 

\noindent Suppose $e\neq r_T$, and write $S$, resp. $S_i$, for $T^\uparrow_e$, resp. $T^\uparrow_{e_i}$, $i=1,\dots,n$. In this case, $\Lambda_u=\chi_u$, where $u\colon \Delta^k{'}\to \cA^S({r_S})$, and similarly $\Lambda_{u_i}=\chi_{u^{S_i}_i}$, where $u^{S_i}_i\colon \Delta^k\to \cA^{S_i}_{r_{S_i}}$. Moreover, since $\chi$ is coherent on nested subtrees of $T$, we have that $\chi_{u^{S_i}_{e_i}}=\chi_{u^S_{i}},$ with $u^S_{i}\colon \Delta^k\to \cA^S({e_i})$. As a consequence, if we denote by $i\colon S \to T$ the subtree inclusion, the diagram \eqref{commut} commutes because it is a compatibility diagram for the map $\chi\colon (S, \alpha i) \to \rho^*(F)$.

\noindent If $e=r_T$ and $u$ is not maximal, let $\overline{u}\colon \Delta^n\to \cA^T({r_T})$ be the maximal extension of $u$ and $d$ the face map for which $\Lambda_u=\lambda_{\overline{u}}\circ d=(\partial\chi)_{{\overline{u}_|}_d}$. If $d$ is not root preserving, let $\{q_1,\dots,q_m\}$ be the leaves of $u_{k'}$, $u\colon \Delta^{k'}\to \cA^T({r_T})$. There are induced maps ${u^{(j)}\colon \Delta^{k'}\to \cA^T({q_j}) }$ for any $j\in \{1,\dots,m\}$, with the property that ${u=(u_{k'})_*\circ (u^{(1)},\dots,u^{(m)})}$ and $ { {(\partial \chi)_{\overline{u}}}_{|_{d}}=\Lambda_u}$, ${\Lambda_u= (u_{k'})_*\circ (\chi_{u^{(1)}},\dots,\chi_{u^{(m)}})}.$ There are subtrees $v_1,\dots,v_m$ of $T$ such that the composition $\cA^T({e_1})\times \dots \times \cA^T({e_n})\to \cA^T({r_T})$ factors as $$\cA^T({e_1})\times \dots \times \cA^T({e_n})\xrightarrow{((v_1)_*,\dots,(v_m)_*)} \cA^T({q_1})\times \dots \times \cA^T({q_m})\xrightarrow{(u_{k'})_*} \cA^T({r_T}),$$ and we can rewrite the diagram in \eqref{prima} as
\begin{center}\begin{tikzcd}
\Delta^k \arrow[d, "v"] \arrow[rr, "{(u_1,\dots,u_n)}"]                            &  & \cA^T({e_1})\times \dots \times \cA^T({e_n}) \arrow[d, "{((v_1)_*,\dots,(v_m)_*)}"] \\
\Delta^{k'} \arrow[rr, "{(u^{(1)},\dots,u^{(m)})}"] \arrow[d, "\id"] &  & \cA^T({q_1})\times \dots \times \cA^T({q_m} )\arrow[d, "(u_{k'})_*"]                   \\
\Delta^{k'} \arrow[rr, "u"]                                                        &  & \cA^T({r_T})                                                                      
\end{tikzcd}
\end{center}
\noindent In particular, commutativity of \Cref{commut} reduces to that of the following diagram

\begin{center}
\begin{tikzcd}
\Delta^k \arrow[d, "v"] \arrow[rr, "{(\chi_{u_1},\dots,\chi_{u_n})}"] &  & F({\alpha(e_1)})\times \dots \times F({\alpha(e_n)}) \arrow[d, "{((v_1)_*,\dots,(v_m)_*)}"] \\
\Delta^{k'} \arrow[rr, "{(\chi_{u^{(1)}},\dots,\chi_{u^{(m)}})}"]     &  & F({\alpha(q_1)})\times \dots \times F({\alpha(q_m)})                                       
\end{tikzcd}
\end{center}
\noindent This diagram is the product of $m$ diagrams, one for each face $T_{q_j}^{\uparrow}$, involving only the operation $v_j$, and each of them is commutative, therefore the product diagram is commutative as well. If $d$ is root preserving and $d(k'-1)=n-1$, the same proof applies mutata mutandis.

\noindent If $d$ satisfies the hypotheses of point (ii) of \Cref{lemmaaux}, we have that $\lambda_u={(\partial \chi)_{\overline{u}}}_{|_{d}}=\chi_{u'}, $ with $u'\colon \Delta^{k'}\to \cA^{T(d,\overline{u})}({r_{T(d,\overline{u})}})$. 

Each of the $u_i\colon \Delta^k\to \cA^T({e_i})$ factors as
$\Delta^k\xrightarrow{u_i'} \cA^{T(d,\overline{u})}({\overline{e}_i})\xrightarrow{\partial}\cA^T({e_i})$, and we have a commutative diagram 
\begin{center}
\begin{tikzcd}
\Delta^k \arrow[d, "v"] \arrow[rr, "{(u_1',\dots,u_n')}"] &  & {\cA^{T(d,\overline{u})}({\overline{e}_1})\times \dots \times \cA^{T(d,\overline{u})}({\overline{e}_n}}) \arrow[d, "(z')_*"] \arrow[rr, "\partial"] &  & \cA^T({e_1})\times \dots\times \cA^T({e_n}) \arrow[d, "\partial(z')_*"] \\
\Delta^{k'} \arrow[rr, "u'"]                              &  & {\cA^{T(d,\overline{u})}({r_{T(d,\overline{u})}})} \arrow[rr, "\partial"]                                                                          &  & \cA^T({r_T})                                                          
\end{tikzcd}
\end{center}

\noindent It follows that $\Lambda_{u_i}=\chi_{u'_i}$ and that we only need to check that ${(z')_*\circ (\chi_{u_1'},\dots,\chi_{u_n'})=\chi_{u'}\circ v}$, which holds as it is an instance of the map $\chi_{|_{\partial}}\colon (T(d,\overline{u}),\alpha \partial)\to \rho^*(F )$.

\item Consider $\partial, w$ and $w'$ as specified in $(2)$; we need to prove that $\Lambda_w=\chi_{w'}$, and for this it suffices to consider the case where $w'$ is non-degenerate as well. 

\noindent If $e\neq r_T$, we have by definition that $\Lambda_w=\chi_{\overline{w}},$ where $\overline{w}\colon \Delta^k\to \cA^{T_e^{\uparrow}}({r_{T_e^{\uparrow}}})=\cA^T(e)$. The restriction of $\partial$ to the subtree $S_{e'}^{\uparrow}$ yields the face map $S_{e'}^{\uparrow}\to T_{e}^{\uparrow}$ and induces a factorization of $\overline{w}$ of the form $\Delta^k\xrightarrow{w'} \cA^{S_{e'}^{\uparrow}}({r_{S_{e'}^{\uparrow}}})=\cA^S({e'})\to \cA^T({e})$, so $\chi_{\overline{w}}=\chi_{w'}$ and the thesis is proven.

\noindent Suppose now that $e=r_T$. Since $w$ factors through a face, it cannot be a maximal chain. Let $\overline{w}\colon \Delta^n \to \cA^T({r_T})$ and $d\colon \Delta^k\to \Delta^n$ respectively be the maximal extension of $w$ and the face map such that $w=\overline{w}\circ d$ and $\Lambda_w=\Lambda_{\overline{w}}\circ d= {(\partial\chi)_{\overline{w}}}_{|_{d}}$. 

\noindent If $d$ satisfies the hypotheses of point (ii) of \Cref{lemmaaux}, then $\Lambda_w=\chi_{w''}$ for a non-degenerate $\Delta^k\xrightarrow{w''} \cA^{T(d,\overline{w})}({r_{T(d,\overline{w})}})$. Since $w'$ is non-degenerate, initiality of $(T(d,\overline{w}), \partial,w)$ yields a face map $T(d,\overline{w})\to S$, and since $\chi$ is coherent on nested faces we have $ \chi_{w''}=\chi_{w'},$ and the thesis is proven. 

\noindent If $d$ is not root preserving, observe that $(\partial\chi)_{{{\overline{w}}_|}_{d}}= (w_{k'})_* \circ (\chi_{w_1},\dots,\chi_{w_n}) $, with $w_i\colon \Delta^k\to \cA^T({e_i})$ and $\{e_1,\dots,e_n\}$ the leaves of the tree $w_{k'}$. In particular, $w$ factors via the operadic composition of $w'_i\colon \Delta^k\to \cA^{S}({\overline{e}_i})$, namely $$w= \partial\circ (w'_{k'})_*(w'_1,\dots,w'_n)= (w_{k'})_*\circ (\partial w'_1,\dots,\partial w'_n)=(w_{k'})_*\circ (w_1,\dots,w_n),$$ so $\chi_{w_i}= \chi_{w'_i}$, hence the thesis. The same proof applies if $d$ is root preserving and $d(k-1)=n-1$.

\item Let $w\colon \Delta^k\to \cA^T(e)$ be non-degenerate. If $e\neq r_T$, then $\Lambda_w=\chi_w$, and it holds by hypothesis that $f_{\alpha(e)}\circ \chi_w= \xi_w$. If $e=r_T$ and $w$ is a maximal chain, by hypothesis on $\Lambda_w=\lambda_w$ we have that $f_{\alpha(r_T)}\circ \Lambda_w=\xi_w$. If $w$ is not maximal, let $\overline{w}\colon \Delta^n\to \cA^T({r_T})$ be a maximal extension and $d\colon \Delta^k\to \Delta^n$ the face such that $w=\overline{w}\circ d$. The compatibility conditions on the collection of $(\xi_u)_u$ ensure that  $\xi_w= \xi_{\overline{w}}\circ d$. By hypothesis we know that $\Lambda_w=\lambda_{\overline{w}}\circ d$ and that $f_{\alpha(r_T)} \circ \lambda_{\overline{w}}= \chi_{\overline{w}}$, so $f_{\alpha(r_T)} \circ \Lambda_{w}= \chi_{w}$, as wanted.
 \end{enumerate}
 \noindent We have checked all the required conditions, so this concludes the proof.
\end{proof}

\subsection{Lifts against leaf and dendroidal inner horns}\label{lift2}

We now provide a similar construction of lifts of morphisms of the form $\rho^* (f)\colon \rho^*(F)\to \rho^* (G)$ against dendroidal leaf and inner horn inclusions. Most of constructions and proofs will be similar to the ones of the previous section, so we will only stress the points where the constructions diverge. In particular, given a tree $T$ and a leaf vertex or an inner edge $x$ of $T$, the horn $\Lambda^x T$ is defined as the colimit on all the faces of $T$ except for $\partial_x$.

\begin{construction}\label{construction2} Let $F$ be a simplicial $P$-algebra and $(T,\alpha)$ a representable of $\dsets/\cN_dP$, and consider a non-degenerate maximal chain $\overline{u}\colon \Delta^n\to \cA^T({r_T})$. Let $x$ be an inner edge or a leaf vertex of $T$, and let $j\colon \Lambda^xT\to T$ be the corresponding dendroidal horn inclusion. Let $k\in [n]$ be the index such that $x$ appears in $\overline{u}_k$ but not in $\overline{u}_{k+1}$; observe that $0\leq k< n$, and necessarily $0<k<n$ when $x$ is an inner edge.

\noindent Given a morphism $\chi\colon (\Lambda^x T, \alpha j)\to \rho^* (F)$, we define a map of simplicial sets ${(\Lambda^n_k \chi)_{\overline{u}} \colon \Lambda^n_k\to F({\alpha(r_T)})}$ by constructing, for any face map $d\colon \Delta^i \to \Delta^n$, $d\neq d^k$, a morphism $X_{\overline{u},d}\colon \Delta^{i}\to F({\alpha(r_T)})$, coherently on nested faces.

\noindent Consider first $d$ satisfying the hypotheses of point (ii) of \Cref{lemmaaux}, and let $(T(d,\overline{u}), \partial, u)$ be the triple given by the lemma; in particular, $u\circ d= \partial\circ \overline{u}$. 
The dendroidal face $\partial\colon T(d,\overline{u})\to T$ is necessarily different from $\partial_x$, hence it is a face in the colimit diagram for $\Lambda^xT$. We then proceed as in \Cref{construction1}, obtaining a morphism $\Lambda^n_k \chi\colon \Lambda^n_k\to F({\alpha(r_T)})$; we can do so because, for any edge $e$ of $T$, the external face map $ T_e^{\uparrow}\to T$ is necessarily different from $\partial_x$.

\noindent If $d$ falls out the hypothese of the lemma, we can define $X_{\overline{u},d}$ by proceding exactly as in \Cref{construction1}.

\noindent The construction is coherent on nested faces, and defines a morphism $(\Lambda^n_k \chi)_{\overline{u}} \colon \Lambda^n_k\to F_{\alpha(r_T)}$, as wanted.
\end{construction}

 \begin{prop}\label{projfibleftfib}
Let $f\colon F \to G$ be a morphism of $P$-algebras and $(T,\alpha)$ a representable in $\dsets/\cN_dP$. Let $x$ be an inner edge or a leaf vertex of $T$. To have a lift $\Gamma$ in the solid diagram
\[
\begin{tikzcd}
(\Lambda^x T, \alpha\circ j) \arrow[d] \arrow[r, "\chi"] & \rho^*F \arrow[d, "\rho^*f"] \\
(T,\alpha) \arrow[r, "\xi"] \arrow[ru, dotted, "\Gamma"] & \rho^*G
\end{tikzcd}
\]

it is sufficient to have a lift  in the diagrams of simplicial sets

\begin{center}
\begin{tikzcd}
\Lambda^n_k \arrow[r, "(\Lambda^n_k \chi)_{u}"] \arrow[d]               & F_{\alpha(r_T)} \arrow[d, "f_{\alpha(r_T)}"] \\
\Delta^n \arrow[r, "\xi_u"'] \arrow[ru, "\gamma_u", dashed] & G_{\alpha(r_T)}                           
\end{tikzcd}
\end{center}  for any non-degenerate maximal chain $u\colon \Delta^n\to \cA^T_{r_T}$, where $0\leq k<n$ is the index for which $x$ appears in $u_k$ but not in $u_{k+1}$.
\end{prop}

\begin{proof}
\noindent The proof of \Cref{reducelift} can be adapted to this context with minor changes; let us explain how.

\noindent Assume we are given lifts $(\gamma_u)_{u}$ as in the statement. We construct, for any edge $e$ of $T$ and any non-degenerate chain $w\colon \Delta^k \to \cA^T(e)$, a map of simplicial sets $\Gamma_w\colon \Delta^k\to F_{\alpha(e)}$ such that:
\begin{enumerate}
\item the family $\{\Gamma_w\}_w$ defines a morphism $\Gamma\colon (\Lambda^xT, \alpha j) \to \rho^*(F)$;
\item $ \Gamma j = \chi$, which means that if $\partial\colon S\to T$ is a dendroidal  face map with $\partial\neq \partial_x$, and $w\colon \Delta^k\to \cA^T(e)$ factors as $\Delta^k\xrightarrow{w'} \cA^S({e'})\to \cA^T({\partial(e')})$ for some $w'$ then $\Gamma_w=\chi_{w'}$;
\item $\rho^* (f)\circ \Gamma=\xi$.
\end{enumerate} 

\noindent Fix such $T,e$ and $w$.
If $e\neq r_T$, then there is a natural identification $\cA^T(e)=\cA^{T_e^{\uparrow}}({r_{T_e^{\uparrow}}}).$ The subtree inclusion $\partial\colon {T_e^{\uparrow}} \to T$ is necessarily different from $\partial_x$, so we define $\Gamma_w\coloneqq \chi_w$.

\noindent If $e=r_T$ and $w$ is not maximal, choose non-degenerate maximal chain $\overline{w}\colon \Delta^n \to \cA^T({r_T})$ and a face map $d\colon \Delta^i\to \Delta^n$ such that $w=\overline{w}\circ d$, and define $\Gamma_w= \gamma_{\overline{w}}\circ d$. Observe that, if $d\neq d^k $, there is an equality $\Gamma_w=\gamma_{\overline{w}}\circ d={(\Lambda^x\chi)_{\overline{w}}}_{|_{d}}$. If moreover $d$ satisfies the hypotheses of point (ii) of \Cref{lemmaaux}, we can write ${(\Lambda^x\chi)_{\overline{w}}}_{|_{d}}=\chi_{w'} $, where  $w'\colon \Delta^i\to \cA^{T(d,\overline{w})}({r_{T(d,\overline{w})} }) $ is the simplex such that  $\overline{w}\circ d= w=\partial\circ w'$ given by the lemma.

\noindent We now prove that the collection $\{\Gamma_w\}_w$ satisfies the three required conditions.
\begin{enumerate}
\item Consider the diagrams as in point $(1)$ of the proof of \Cref{reducelift}. If $e\neq r_T$, the thesis can be proven using the same arguments. If $e=r_T$ and $u$ is not maximal, let $\overline{u}\colon \Delta^n\to \cA^T_{r_T}$ be the maximal extension of $u$ such that $u= \overline{u}\circ d$ and $\Gamma_u=\gamma_{\overline{u}}\circ d= {(\Lambda^n_k \chi)_{\overline{u}}}_{|_{d}}.$ If $d\neq d^k$, we  proceed as in the analogous case of \Cref{reducelift} (the dendroidal face map $\partial$ is different from $\partial_x$ when $d$ is root preserving). Consider now the case $d=d^k\colon \Delta^{n-1}\to \Delta^n$. If the face map $v\colon \Delta^i\to \Delta^{n-1}$ is not the identity, we have that $\Gamma_u\circ v= \gamma_{\overline{u}}\circ d \circ v= {(\Lambda^n_k\chi)_{\overline{u}}}_{|_{d\circ v}}=\chi_{u'}$ for $u'\colon \Delta^i\to \cA^{T(d\circ v, \overline{u})}({r_{T(d\circ v, \overline{u})}})$, so we are back to checking a condition on $\chi$ involving a face map $\partial\colon T(d v, \overline{u})\to T$, which lies in the colimit diagram for $\Lambda^xT$. If $v=\id$, we conclude that necessarily $n=1$, $\cA^T({e_1})= \cA^T({r_T})$ and the map $\cA^T({e_1})\to \cA^T({r_T})$ is the identity, so the thesis holds trivially.

\item Consider a dendroidal face $\partial\colon S\to T$, $\partial\neq \partial_x$, and a simplex $w\colon \Delta^i\to \cA^T(e)$ which factors as $\Delta^i\xrightarrow{w'}\cA^S({e'})\xrightarrow{\partial}\cA^T({\partial(e')})$ for some $e'\in E(S)$ with $\partial(e')=e$. We can reason as in the proof of \Cref{reducelift}: if $e\neq r_T$ we prove that $\Gamma_w=\gamma_{w'}$; if $e=r_T$, we can still apply the arguments because, if $\Gamma_w= {(\partial\chi)_{\overline{w}}}_{|_{d}}=\gamma_{\overline{w}}\circ d$ for a maximal extension $(\overline{w},d)$ of $w$, necessarily $d\neq d^k$, since otherwise $\partial=\partial_x$.

\item This point can be proven in a completely analogous way as point $(3)$ in the proof of \Cref{reducelift}.
\end{enumerate}
\noindent This concludes the proof.
\end{proof}

\subsection{The Quillen adjunction}

\noindent We are ready to prove the main result of this section. 

\begin{theorem}\label{adjj}
	The straightening-unstraightening adjunction $$\adjunction{\rho_!}{\dsets/\cN_d P}{\alg_P(\ssets)}{\rho^*} $$ is a Quillen adjunction between the covariant model structure on $\dsets/\cN_d P$ and the projective model structure on $\alg_P(\ssets).$
\end{theorem}

\begin{proof}
Recall that a morphism $(X,f)\to (Y,g)$ in $\dsets/\cN_d P$ is a left fibration, resp. trivial fibration, if it has the right lifting property against inner and leaf horn inclusions $(\Lambda^xT, \alpha\circ j)\to (T,\alpha)$, resp. boundary inclusions $(\partial T, \alpha\circ \iota)\to (T,\alpha)$. A morphism of $P$-algebras $F\to G$ is a projective fibration, resp. projective trivial fibration, if for any object $c$ of $P$, the map of simplicial sets $F(c)\to G(c)$ has the right lifting property with respect to all horn inclusions $\Lambda^n_k\to \Delta^n$, resp. to the boundary inclusions $\partial\Delta^n\to \Delta^n$.

\noindent After \Cref{reducelift}, the functor $\rho^*$ sends covariant trivial fibrations to projective trivial fibrations, and \Cref{projfibleftfib} shows that it sends projectively fibrant objects to left fibrations. As covariant fibrations between fibrant objects are left fibrations, this means that $\rho^*$ preserves fibrations between fibrant objects. It is a standard model-categorical fact that this is enough to ensure it preserves all fibrations.

\noindent We have shown that $\rho^*_P$ is a right Quillen functor, and this concludes the proof.
\end{proof}

\section{Monoidal rectification of left fibrations}\label{section6}

\noindent Consider $P=A$ a discrete category. The rectification functor and the relative dendroidal nerve for $A$ yield an adjunction $$ \adjunction{\rho_!^A}{\ssets/\cN A}{\mathsf{Fun}(A,\ssets)}{\rho^*_A}$$ which is a Quillen adjunction with respect to the covariant  model structure on the left hand side and the projective model structure on the right hand side. As observed in \Cref{recover}, this adjunction recovers the one defined in \cite{HM:LFHC}, where it is proven that it is a Quillen equivalence.

\noindent In this section, we 
\begin{enumerate}
\item give a new independent proof that $(\rho_!^A,\rho^*_A)$ is a Quillen equivalence for a discrete category $A$. Contrarily to the one in \cite{HM:LFHC}, it does not rely on comparing $\rho_!^A$ with the homotopy colimit functor;
\item improve the above result by showing that, when $A$ is a symmetric monoidal category, the Quillen equivalence $(\rho_!^A,\rho^*_A)$ is in fact a monoidal Quillen equivalence of monoidal model categories. 
\end{enumerate}

\noindent In \Cref{section8} we will actually prove that $(\rho_!^P,\rho^*_P)$ is a Quillen equivalence for \emph{any} $\Sigma$-free discrete operad $P$, but that will require to use $\infty$-categories, while in this section we are still able to prove our results with model categorical and point-set level methods.

\subsection{The Quillen equivalence: discrete categories}\label{section6.1}

\noindent Fix a discrete category $A$ and write $\rho_!$, resp. $\rho^*$, for $\rho_!^A$, resp. $\rho^*_A$. Consider a simplicial set $X$ and an element $(X,f)$ in $\ssets/\cN A$. The element $(X,\alpha)$ is the colimit of representables, $$(X,f)\simeq \underset{{([n],\alpha)\to (X,f)}}{\colim ([n],\alpha)}.$$ \noindent As colimits of functors can be computed objectwise, for every object $a$ of $A$ one has $$ \rho_!(X,f)\simeq \underset{([n],\alpha)\to(X,f)}{\colim}\cN(\alpha/a).$$ \noindent Now, the nerve of the poset $\alpha/a$ can be written as $$ \cN(\alpha/a)\simeq \Delta^n \times_{\cN A}\cN (A/a),$$ hence \begin{equation}
	\rho_!(A,\alpha)\simeq X \times_{\cN A} \cN(A/a).
\end{equation}

\noindent With this description in hand, we can prove that the Quillen adjunction $(\rho_!,\rho^*)$ induces an equivalence of homotopy categories. The credits of the following proof go to Gijs Heuts.

\begin{theorem}\label{bonobo}
	For any discrete category $A$, the Quillen adjunction $$\adjunction{\rho_!}{ \ssets/\cN A}{ \mathsf{Fun}(A,\ssets)}{\rho^*}$$ is a Quillen equivalence between the covariant and the projective model structure.
\end{theorem}

\begin{proof}
	
	\noindent It is a standard result in model category theory that it is enough to prove that $\rho^*$ reflects weak equivalences between fibrant objects and that the derived unit evaluated at a bifibrant object is a weak equivalence. Let us prove these two facts.
	
	\noindent Let $\varphi\colon F \to G$ be a morphism between Kan-enriched functors. Since $\rho^*$ is right Quillen, both $\rho^*(F)$ and $\rho^* (G)$ are left fibrations, which means that $\rho^*(\varphi)$ is a covariant weak equivalence if and only if it is a fibrewise weak homotopy equivalence of spaces. Given an object $a$ of $A$, after \Cref{freegaza} there is a natural  equivalence between the map of fibres
	\[ \rho^*(\varphi)_a  \colon   \rho^*(F)_a \longrightarrow \rho^*(G)_a\]
and the map between the evaluations
\[	\varphi_a \colon F(a) \longrightarrow  G(a).\] Therefore $\rho^*(\varphi)$ is a weak equivalence if and only if $\varphi$ is, so in particular $\rho^*$ reflects weak equivalences.

	\noindent For the second condition, observe that every object in $\ssets/\cN A$ is cofibrant, so let $(X,f)$ be a left fibration and $\mathbb{L}\eta_{(X,\alpha)}$ be the derived unit evaluated at $(X,f)$. This can be written as the composition $$ \mathbb{L}\eta_{(X,f)}\colon (X,f) \xrightarrow{\eta_{(X,f)}} \rho^*\rho_! (X,f) \xrightarrow{\rho^*(\cR)} \rho^* (F),$$ where $\cR\colon \rho_!(X,f)\xrightarrow {\sim}F$ is a fibrant replacement for $\rho_!(X,\alpha)$. As both domain and codomain are left fibrations and weak equivalences between left fibrations are fibrewise weak homotopy equivalences, the derived unit is a weak equivalence if and only if it is fibrewise so. Consider an object $a$ of $A$; as $\rho^*(\cR)_a\simeq \cR_a$ and $\cR$ is a projective weak equivalence, it is sufficient to show that the map of fibres  $(X,f)_a\to \rho^*(\rho_!(X, f))_a \simeq \rho_!(X,f)(a)$ given by the non-derived unit is a weak homotopy equivalence of simplicial sets. This map appears as the top horizontal arrow in the pullback
	\[
	\begin{tikzcd}
		(X, f)_a  \arrow[d] \arrow[r] & X \times_{\cN A} \cN (A/a) \arrow[d]\\
\Delta^0 \arrow[r, "\{\id_a\}"] & \cN (A/a) \ .
	\end{tikzcd}
	\]

\noindent The map $\Delta^0\xrightarrow{\{\id_a\}} \cN (A/a)$ is \emph{right anodyne}, which means it belongs to the closure under pushouts, retracts and transfinite compositions of the set of right horn inclusions ${\{\Lambda^n_k\to \Delta^n\}_{n, 0<k\leq n}.}$ As left fibrations are stable under pullbacks, also the map $X\times_{\cN A}\cN(A/a) \to \cN(A/a) $ is a left fibration. As proven in \cite[Proposition 2.10]{HHR:SPST}, the pullback of a right anodyne map along a left fibration is again right anodyne. As right anodyne maps are in particular weak homotopy equivalences of simplicial sets, this concludes the proof.
\end{proof}

\subsection{The monoidal Quillen equivalence}\label{section6.2}

\noindent Consider now a symmetric monoidal category $A=(A,\otimes,1_A)$. As simplicial sets are monoidal with respect to cartesian product, the category $\mathsf{Fun}(A,\ssets)$ has a symmetric monoidal structure given by \emph{Day convolution}: given functors $F,G\colon A \to \ssets$, their tensor product $F\underset{Day}{\otimes} G \colon A \to \ssets$ is the two-variables left Kan extension of the product of $F$ and $G$ $$ A\times A \longrightarrow \ssets \ , \ (a,b)\mapsto F(a)\times G(b)$$ along the tensor product of $A$, so that $$ \left(F\underset{Day}{\otimes} G\right) (c)\simeq \underset{a\otimes b \to c}{\colim}\ F(a)\times G(b).$$ \noindent The unit is given by the Yoneda embedding of the unit of $A$, seen as a discrete simplicial presheaf on $A$.

\noindent The over-category $\ssets/\cN A$ also has a symmetric monoidal structure $\boxtimes$, defined as $$\boxtimes\colon \ssets/\cN A\times \ssets/\cN A \xlongrightarrow{-\times-} \ssets/(\cN A\times \cN A )\xlongrightarrow{m_!} \ssets/\cN A,$$ where $m_!$ is the post-composition with the nerve of the tensor product of $A$. The unit for the monoidal structure on $\ssets/\cN A$ is given by the map $(\ast, \mathbbm{1}_A\colon \ast \xrightarrow{\{1_A\}} \cN A)$ selecting the tensor unit of $A$.

\noindent A \emph{monoidal Quillen model category} is a symmetric monoidal category $(M,\otimes,1_M)$ endowed with a model category structure and whose tensor product which is appropriately compatible with the model category structure, ensuring, for instance, that the homotopy category is also symmetric monoidal. In particular, in a monoidal model category the tensor product has to be a \emph{left Quillen bifunctor}: given two cofibrations $i\colon X\to Y$, $j\colon U \to V$, their \emph{pushout product} $$i\square j\colon X\otimes V \bigcup_{X\otimes U} Y \otimes U \longrightarrow Y\otimes V$$ is a cofibration, which is also a weak equivalence if $i$ or $j$ is.

\begin{prop}[{\cite[Proposition 2.2.15]{I:CHTMMC}}]
Let $(A,\otimes,1_A)$ be a symmetric monoidal category. The  projective model structure on $\mathsf{Fun}(A,\ssets)$ is model monoidal with respect to Day convolution.
\end{prop}

\noindent The covariant model structure on $\ssets/\cN A$ is also model monoidal. In order to show this, let us recall that

\begin{prop}\label{mon1}
	Let $(A,\otimes,1_A)$ be a symmetric monoidal category. The covariant model structure on the symmetric monoidal category $\left(\ssets/\cN A,\boxtimes, \mathbbm{1}_A \right)$ is model monoidal.
\end{prop}
\begin{proof}
	\noindent The functor $m_!$ is a left Quillen functor (\cite{HeMo:SDHT}), and as the unit of the tensor product is cofibrant, the statement reduces to showing that $-\boxtimes-$ is a left Quillen bifunctor. It is clearly cocontinuous in both variables, and it is also easy to see that the pushout product of two cofibrations is a cofibration as well. Consider  a trivial cofibration $i\colon (X,f) \to (Y,g)$; its pushout-product $i\square j$ with $j\colon (U,h)\to (V,k)$ is the morphism appearing in the following diagram
	\[
	\begin{tikzcd}
		(X,f)\times (U,h) \arrow[d, "i\times \id"'] \arrow[r] & (X,f)\times (V,k) \arrow[d] \arrow[rdd, "i\times \id", bend left] &           \\
		(Y,g)\times (U,h) \arrow[r] \arrow[rrd, bend right]     & \bullet \arrow[rd, "i\square j"]                            &           \\
		&                                                             & (Y,g)\times (V,k)
	\end{tikzcd}
	\]
	
		\noindent By \cite[Lemma 3.3]{NS:PSMICRSMMC}, for any simplicial set $X$ and any covariant weak equivalence $\varphi\colon S\to T$ in $\ssets/X$, the functor  $$ \varphi\times - \colon \ssets/Y \longrightarrow \ssets/X\times Y$$ is left Quillen with respect to the \emph{contravariant} model structure, which is obtained from the covariant model structure by applying the duality $\Delta\to\Delta, \ [n]\mapsto [n]^{\text{op}}$. It follows that $\varphi\times - $ is left Quillen also with respect to the covariant model structure.
	
	\noindent In particular, $i\times \id$ is a weak equivalence; since the covariant model structure is also left proper, the morphism $(X,f)\times (V,k) \to \bullet$ is a weak equivalence as well, and by the $2 $-out-of-$3$ property $i\square j$ is a weak equivalence as well. This concludes the argument.
\end{proof}

\begin{theorem}\label{kiku}
	If $A$ is a symmetric monoidal category, the Quillen adjunction $$\adjunction{\rho_!}{\ssets/\cN A }{\mathsf{Fun}(A,\ssets)}{\rho^*}$$ is a monoidal Quillen equivalence of monoidal model categories for the above monoidal model structures.
\end{theorem}
\begin{proof}
	After \Cref{bonobo}, the adjunction  $(\rho_!,\rho^*)$ is a Quillen equivalence, so we only need to prove $\rho_!$ is monoidal. Write $X\coloneqq \cN A$, and fix two elements $(U,h)$ and $(V,k)$ in $\ssets/X$. We prove that $\rho_!$ is monoidal in three steps.

	\begin{enumerate}
		
		\item  First, we prove that $\rho_!$ is lax monoidal. We construct the laxity map $$ \rho_!(U,h)\otimes_{\text{Day}} \rho_!(V,k)\longrightarrow \rho_!((U,h) \boxtimes (V,k))$$ as follows. Fix an object $a$ of $A$ and a map $\phi\colon b\otimes c \to a$ for objects $b$ and $c$. There is a commutative diagram
		\[
		\begin{tikzcd}
			(U\times V)\underset{\cN A \times \cN A}{\times} \left( \cN (A/a)\times \cN(A/b) \right) \arrow[d, "\text{pr}_2"] \arrow[r, "\text{pr}_1"] & U\times V \arrow[rr,"\beta\times\gamma"] & &\cN A\times \cN A \arrow[d, "\otimes"]  \\
			\cN (A/a)\times \cN(A/b) \arrow[r,"m_!"] & \cN (A/a\otimes b) \arrow[r, "\phi_*"] & \cN (A/c) \arrow[r] & \cN A
		\end{tikzcd}
		\]
	and they induce the map	$$  \left(\rho_!(U,h)\underset{\text{Day}}{\otimes} \rho_!(V,k) \right)(c) \simeq \underset{a\otimes b\to c}{\colim} (U\times V)\underset{\cN A \times \cN A}{\times} \left( \cN (A/a)\times \cN(A/b) \right)\longrightarrow (U\times V)\times_{\cN A} \cN (A/	c),$$ as wanted.
		
		\item The functor $\rho_!$ is also colax monoidal, and the colaxity map $$ \rho_!((U,h)\boxtimes (V,k))\longrightarrow \rho_!(U,h)\underset{\text{Day}}{\otimes} \rho_!(V,k)$$ is constructed as follows. Let $c$ be an object in $A$ and let $n\geq 0$. A $n$-simplex $X$ of ${\rho_!((U,h)\boxtimes(V,k))(c)}$ corresponds to a tuple $((a,b), x)$, where $a\in U_n$, $b\in V_n$, $x\in \cN(A/c)_n$, satisfying the property that $(h\boxtimes k)_n (a,b)= \cU(x)$, where $\cU\colon \cN(A/c)\to \cN(A)$ is the nerve of the forgetful functor.	In other words, if we write $$h_n(a)=a_0\to \dots \to a_n, \quad  k_n(b)=b_0\to \dots b_n, \text{ and } {x =c_0\to c_1\to \dots \to c_n \to c} \quad ,$$ the property that $((a,b),x)$ satisfies is that $$ \text{ $c_i= a_i\otimes b_i$ and $c_i\to c_{i+1}= (a_i\to a_{i+1})\otimes ( b_i\to b_{i+1})$ }$$  for all $i$ for which it makes sense.
		
		\noindent In particular, the couple $$( (a, a_0\to \dots \to a_n \xrightarrow{id}a_n), (b, b_0\to \dots \to b_n \xrightarrow{id} b_n))$$ determines a $n$-simplex of $$\left(U\underset{\cN A}{\times} \cN (A/a_n)\right)\times \left(V\underset{\cN A}{\times}\cN(A/b_n)\right)$$ which determines a $n$-simplex of the colimit $ (\rho_!(h)\underset{\text{Day}}{\otimes}\rho_!(k))(c)$ relative to the component determined by  the morphism ${(c_n=a_n\otimes b_n)\to c}$.
		\item $\rho_!$ is strong monoidal, because the laxity and colaxity map of, resp., points $(1)$ and $(2)$, are one the inverse of the other. That the composite $$\rho_!((U,h)\boxtimes(V,k))\longrightarrow\rho_!(U,h)\underset{\text{Day}}{\otimes} \rho_!(V,k)\longrightarrow \rho_!(h\boxtimes k)$$ is the identity is immediate to check; to see that the composite $$\rho_!(U,h)\underset{\text{Day}}{\otimes} \rho_!(V,k)\longrightarrow \rho_!((U,h)\boxtimes(V,k))\longrightarrow \rho_!(U,h)\underset{\text{Day}}{\otimes} \rho_!(V,k)$$ is the identity as well, it suffices to notice that, given a morphism $\varphi\colon d_1\otimes d_2\to c $ (using the same notations as in the previous point), the $n$-simplices 
		$$((a, a_0\to \dots \to a_n \to d_1), (b, b_0\to \dots\to b_n \to d_2 )) \text{ in $\rho_!(U,h)(d_1)\times \rho_!(V,k)(d_2)$ }$$ and $$ ((a, a_0\to \dots \to a_n \xrightarrow{\id} a_n), (b, b_0\to \dots \to b_n \xrightarrow{\id} b_n)) \text{ in $\rho_!(U,h)(a_n)\times \rho_!(V,k)(b_n)$ }$$ are identified under Day convolution.
	\end{enumerate}
	
\end{proof}

\section{The Quillen equivalence}\label{section8}

\noindent After \Cref{section6}, we know that for a discrete category $A$ and an element $(M,f)$ of $\ssets/\cN A$, for every object $a$ of $A$ one has $$ \rho^A_!(M,f)(a)\simeq X \times_{\cN A} \cN (A/a).$$

\noindent When we consider a discrete operad $P$, explicitely describing the simplicial $P$-algebra $\rho_!^P(X,f)$ for $(X,f)$ not representable in $\dsets/\cN_d P$ is a hard procedure, essentially because (non-filtered) colimits in $\alg_P(\ssets)$ cannot be computed object-wise when $P$ has also non-unary multimorphisms. 

\noindent The solution we adopt is to move to the language of $\infty$-categories, relating the functor of $\infty$-categories induced by the rectification functor (which is left Quillen) with the operadic straightening functor $\mathsf{St}^P$ defined in \cite{P:SUEIO}. In addition, we obtain a \emph{up-to-homotopy} explicit description of the simplicial $P$-algebra $\rho_!^P(X,f)$ for any $(X,f)$ in $\dsets/\cN_d P$.

\noindent In the next sections, we assume the knowledge of $\infty$-categories (\cite{Lu:HTT}) and Lurie's $\infty$-operads and symmetric monoidal $\infty$-categories as defined in \cite{Lu:HA}. We refer to \cite[\S 1]{P:SUEIO} for a concise and reasonably self-contained account of these notions.

\subsection{From model categories to $\infty$-categories}\label{backforth}\label{sofar}

\noindent Let us concisely reecall how to translate model categorical statements into statements of $\infty$-categories.

\begin{itemize}
	\item To any model category $\cM$ one can associate its \emph{underlying $\infty$-category $\cM_\infty$}. Explicitly, $\cM_\infty$ can be obtained as the homotopy coherent nerve of a fibrant replacement of the Dwyer-Kan localization of $\cM$ (\cite{DK:CSL}); in this model, $\cM$ and $\cM_\infty$ have the same set of objects. We say that an $\infty$-category $X$ is \emph{presented by} a model category $\cM$ if there exists an equivalence $X\simeq \cM_\infty$.
 	\item Any left (resp. right) Quillen functor $F\colon \cM\to \cM'$ induces a left (resp. right) adjoint functor of $\infty$-categories $F_\infty\colon\cM_\infty \to \cM'_\infty$ (\cite[Theorem 2.1]{MG:QAIAQC}). On objects, it can be computed as $F_\infty(X)=\mathbb{L}F(X)=F(X^{\text{cof}})$ (resp. $F_\infty(X)=\mathbb{R}F(X) \simeq F(X^{\text{fib}})$), where $\mathbb{L}F$, resp. $\mathbb{R}F$, denotes the left, resp. right, derived functor of $F$, and $X^{\text{cof}}$, resp. $X^{\text{fib}}$, is a cofibrant, resp. fibrant, replacement of $X$ in $\cM$. Quillen embeddings correspond to fully faithful left adjoints, and any Quillen equivalence induces an equivalence of $\infty$-categories, as follows from \cite[\S A.2]{MG:QAIAQC} and \cite[\S 1.3.4]{Lu:HA}.
	
	\item  Given a functor $\cG$ of $\infty$-categories and a Quillen functor $F$ of model categories, we say that \emph{$F$ presents $\cG$} if there is an equivalence $\cG \simeq F_\infty$.
	
	\item There is an equivalence of functors $\mathbb{L}F \simeq \mathsf{ho}(F_\infty)$ where $\mathsf{ho}(-)$ is the homotopy category functor, left adjoint of the nerve ${\Cat\to \ssets.}$
	
	\item  Any monoidal Quillen model category $\cM$ yields a symmetric monoidal $\infty$-category $\cM^\otimes_{\infty} $, whose underlying $\infty$-category is $\cM_\infty$ (\cite{H:DKLR}).
	
	\item A functor of $\infty$-categories $L\colon \cC\to\cD$ is a \emph{localization} if it has a fully-faithful right adjoint. In particular, any localization functor between presentable $\infty$-categories is cocontinuous and essentially surjective.
	
	\item If $L\colon \cM \to \cN$ is a left Bousfield localization of Quillen model categories, the induced functor of $\infty$-categories $L_\infty\colon \cM_\infty \to \cN_\infty$ is a localization. 

\end{itemize}

\noindent Also, let us introduce some

\begin{notation}\hfill
	\begin{itemize}
		\item We write $\cS$ for the $\infty$-category of spaces, presented by the Kan-Quillen model structure on simplicial sets.
		\item Given a quasioperad $X$, we write $\mathsf{dLeft}_X$ for the $\infty$-category of dendroidal left fibrations over $X$, presented by the covariant model structure on $\dsets/X$. Observe that, if $X=\cC$ is in fact a quasicategory, there is an isomorphism of covariant model structures $\dsets/\cC\simeq \ssets/\cC$, and we just write $\mathsf{Left}_\cC$ for the $\infty$-category it presents.
		\item  Given a discrete operad $P$, we write $\alg_P(\cS)$ for the $\infty$-category presented by the projective model structure on $\alg_P(\ssets)$. It is a consequence of \cite[Theorem 7.11]{PS:ARCSO} that, when $P$ is also $\Sigma$-free, the $\infty$-category $\alg_P(\cS)$ is equivalent to the $\infty$-category of $P$-algebras in Lurie's formalism \cite{Lu:HA}.
		\item We write $\mathsf{dOp}_\infty$ for the $\infty$-category presented by the model structure for quasioperads on $\dsets$.
	
	\end{itemize}
\end{notation}

\noindent The equivalences proven in \Cref{section6} can be formulated as follows.

\begin{coro}[\Cref{bonobo}]\label{coro1}
	Let $A$ be a discrete category $A$. The rectification functor and the relative nerve functor are mutually inverse in an equivalence of $\infty$-categories $$\adjunction{(\rho_!^A)_\infty}{\mathsf{Left}_{\cN A}}{\mathsf{Fun}(A, \cS)}{(\rho^*_A)_\infty}.$$ 
\end{coro}

\begin{coro}[\Cref{kiku}]\label{coro2}
	Let $A$ be a discrete symmetric monoidal category. There is an equivalence of symmetric monoidal $\infty$-categories $$ \adjunction{(\rho_!^A)^\otimes_\infty}{(\mathsf{Left}_{\cN A})^\otimes}{\mathsf{Fun}( A, \cS)^\otimes}{(\rho^*_A)^\otimes_\infty}, $$ which on the underlying categories coincides with the equivalence in \Cref{coro1}.
	
\end{coro}

\noindent In particular, the Quillen equivalence proven in \Cref{kiku} is a presentation of the monoidal un/straightening equivalence of \cite{R:MGCIC} for $A$ a discrete symmetric monoidal $\infty$-category. In particular, we can use the explicit formula for $\rho_!^A$ to compute the monoidal straightening functor.

\begin{coro}[of \Cref{adjj}]
	Let $P$ be a discrete operad. The Quillen adjunction $(\rho_!^P, \rho^*_P)$ induces an adjunction of $\infty$-categories $$ \adjunction{(\rho_!^P)_\infty}{\mathsf{dLeft}_{\cN_d P}}{\alg_P(\cS)}{(\rho^*_P)_\infty}.$$ 
\end{coro}

\noindent We will prove that, when the discrete operad $P$ is moreover $\Sigma$-free, the functor $(\rho_!^P)_\infty$ is equivalent to the operadic straightening functor $\mathsf{St}^P$ of \cite{P:SUEIO}, whose definition we now recall.

\subsection{The operadic straightening functor}\label{panico}

\noindent In \cite{P:SUEIO}, we studied the $\infty$-category of space-valued algebras over \emph{any} $\infty$-operad $X$ by means of a pair of adjoint functors $(\mathsf{St}^X,\mathsf{UnSt}^X)$. In this section, we provide the defiition of $\mathsf{St}^X$ when $X$ is equivalent to a $\Sigma$-free discrete operad $P$ and avoid to go through all the formalism necessary for the more general definition. For this purpose, let us recall the definition of
\subsubsection{The symmetric monoidal envelope functor}

\noindent Let $\Op$ be the category of discrete operad and let $\mathsf{smCat}$ be the category of symmetric monoidal categories with strong monoidal functor. 

\begin{deff}
	The \emph{discrete symmetric monoidal envelope} is the left adjoint to the forgetful functor $U\colon \mathsf{smCat}\to \Op$. Explicitly, given a discrete operad $P$, the symmetric monoidal category $(\mathsf{env}(P),\boxplus, \emptyset)$ is defined as follows:
	\begin{itemize}
		\item the objects are given by the strings $(c_1,\dots,c_n)$ of objects of $P$ with $n\geq 0$, where the empty string is represented by $\emptyset$;
		\item a morphism $(c_1,\dots,c_n)\to (d_1,\dots,d_m)$ is given by a pair $(f,\{p_j\}_{j=1}^m)$, where $f\colon \{1,\dots,n\}\to \{1,\dots,m\}$ is a function of finite sets and $p_j$ is an operation in $P(\{c_i\}_{f(i)=j};d_j)$. Composition comes from the operadic composition of $P9$.
		\item The tensor product $\boxtimes$ is the concatenation of sequences, and the unit is the empty sequence.
	\end{itemize}
	\noindent We denote by $\env(P)$ the nerve of the underlying category of $\mathsf{env}(P)$.
\end{deff}

\noindent This construction generalizes to the context of $\infty$-operads (\cite{Lu:HA}), in the following sense. Let $\ell\mathsf{Op}_\infty$ be the $\infty$-category of Lurie's $\infty$-operads and $\mathsf{smCat}_\infty$ the $\infty$-category of symmetric monoidal $\infty$-categories, where morphisms are strong monoidal functors. There is a forgetful functor $U\colon \mathsf{smCat}_\infty\to \ell\Op_\infty$, and, after \cite{Lu:HA}, it admits a left adjoint$$\env(-)^\otimes \colon \ell\Op_\infty \to \mathsf{smCat}_\infty,$$ called the \emph{symmetric monoidal envelope functor}. It is compatible with the discrete symmetric monoidal envelope, in the sense that there is a commutative diagram of $\infty$-categories

\begin{center}
	\begin{tikzcd}
		\ell\Op_\infty \arrow[r, "\env(-)^\otimes"]   & \mathsf{smCat}_\infty          \\
		\Op \arrow[u, hook, "(-)^\otimes"] \arrow[r, "\mathsf{env}(-)"] & \mathsf{smCat} \ . \arrow[u, hook, "(-)^\otimes"']
	\end{tikzcd}
\end{center}

\noindent where $(-)^\otimes$ denotes the canonical full embeddings of discrete operads and discrete symmetric monoidal categories in their $\infty$-categorical counterpart.

\noindent In particular, for any $\infty$-operad $X$, we denote by $\env(X)$ the underlying $\infty$-category of the symmetric monoidal $\infty$-category $\env(X)^\otimes$.

\subsubsection{Another dendroidal model}\label{doechii}
\noindent Let $\mathsf{PSh}(\omg)$ denote the $\infty$-category of functors $\mathsf{Fun}(\cN(\omg)^{\text{op}},\cS)$. It can be realized, for instance, as the $\infty$-category presented by the projective model structure on simplicial presheaves on $\omg$ with respect to the Kan-Quillen model structure on simplicial sets.

\noindent An element $X$ in $\mathsf{PSh}(\omg)$ is called a complete dendroidal Segal space if it satisfies the dendroidal Segal condition and the completeness condition, as detailed for instance in \cite[\S 2.2.3]{HM:OELIODIO}; write $\mathsf{DOp}_\infty$ for the full sub $\infty$-category of $\mathsf{PSh}(\omg)$ spanned by the complete dendroidal Segal spaces. Complete dendroidal Segal spaces form another model for $\infty$-operads: after \cite{CM:DSSIO}, there is an equivalence of $\infty$-categories $$\adjunction{d_!}{ \mathsf{dOp}_\infty}{\mathsf{DOp}_\infty}{d^*}.$$ It has the property that, given a discrete operad $Q$, one has $d_!(\cN_d Q)\simeq \mathsf{disc}(\cN_d Q)$, where $\mathsf{disc}\colon \mathsf{Sets}\to \cS$ is the inclusion of the $1$-category of sets into the $\infty$-category of spaces.

\noindent Given a complete dendroidal Segal space $X$, we write $\mathsf{DLeft}_X$ for the full sub $\infty$-category of ${\mathsf{DOp}_\infty}_{/X}$ spanned by the objects $(Y,g)$ which are local with respect to dendroidal left and inner horn inclusions. This $\infty$-category is presented by a model structure on $\mathsf{Fun}(\omg^\text{op},\ssets)/X$ analogous to the covariant model structure (and usually called by the same name), which is compatible with the one relative to dendroidal sets in the following sense.

\begin{theorem}[\cite{CM:DSSIO}]\label{cisimo}
	Let $X$ be a quasioperad. There is a commutative diagram of $\infty$-categories
	\[\begin{tikzcd}
		{\mathsf{d}\Op_\infty}_{/X} \arrow[d]\arrow[r, "d_!"] & {\mathsf{D}\Op}_{d_! X} \arrow[d]\\
		\mathsf{d}\Left_X  \arrow[r] & \mathsf{D}\Left_{d_! X}
	\end{tikzcd}
	\] where the horizontal arrows are equivalences, and the vertical arrows are localizations.
\end{theorem}

\subsubsection{The definition}

\noindent Composing the equivalence between the operadic model structure on dendroidal sets and the model structure on dendroidal spaces for complete dendroidal Segal spaces with Hinich-Moerdijk comparison functors of \cite{HM:OELIODIO}, one obtains an equivalence of $\infty$-categories $$\adjunction{\lambda'}{\mathsf{dOp}_\infty}{\Op_\infty}{\delta'}$$ between the $\infty$-category presented by the operadic model structure on dendroidal sets and that of Lurie $\infty$-operads. It enjoys the property that, for a discrete operad $Q$, one has $\lambda'(\cN_d Q)\simeq Q^\otimes$.

\noindent We can now give the following

\begin{deff}
Let $P$ be a $\Sigma$-free discrete operad. The \emph{operadic straightening functor} for $P$ is the functor of $\infty$-categories $$ \mathsf{St}^P\colon \mathsf{dLeft}_{\cN_d P}\longrightarrow \alg_P(\cS)$$ defined in the following way: given a dendroidal left fibration $(X,f)$ and an object $c$ of $P$, the value at $c$ of the $P$-algebra $\mathsf{St}^P(X,f)$ is weakly equivalent to $$ \mathsf{St}^P(X,f)(c)\simeq  \env(\lambda' X)\times_{\env(P)} {\env(P)}_{/c}.$$
\end{deff}

\subsection{The dendroidal rectification presents the operadic straightening }\label{fine}

\noindent Let us prove the statement of the title, and more precisely the following

\begin{theorem}\label{desame}
	For a $\Sigma$-free discrete operad $P$, there is an equivalence of functors $$(\rho_!^P)_\infty \simeq \mathsf{St}^P\colon \mathsf{dLeft}_{\cN_d P}\longrightarrow \alg_P(\cS).$$
\end{theorem}

\begin{proof}

	\noindent Let us denote by $X$ the complete dendroidal Segal space $d_!(\cN_d P)$. In light of \Cref{cisimo}, we can make two reductions: we can write both $(\rho_!^P)_\infty $ and $\mathsf{St}^P$ as functors on the $\infty$-category $\mathsf{DLeft}_X$, and since this latter is a localization of the over-category ${\mathsf{DOp}_\infty}_{/X} $, it is sufficient to show that $(\rho_!^P)_\infty $ and $\mathsf{St}^P$ agree after precomposition with the localization map, that is as functors ${\mathsf{DOp}_\infty}_{/X} \to \alg_P(\cS)$.
	
	\noindent  The functor ${\mathsf{DOp}_\infty}_{/X}\to \mathsf{DLeft}_X$ is cocontinuous and essentially essentially surjective: as the $\infty$-category $\mathsf{PSh}(\omg)/X\simeq \mathsf{PSh}(\omg/X)$ is generated under colimits by the representables (\cite[Corollary 5.1.5.8.]{Lu:HTT}), the $\infty$-category ${\mathsf{DOp}_\infty}_{/X}$ is also (non-freely) generated under colimits by elements of the form  $(T,\alpha),$ where $T$ is a tree and $\alpha\colon T \to P $ is a morphism of discrete operads. Given that both $\mathsf{St}^P$ and $(\rho_!^P)_\infty$ are cocontinuous, they are equivalent if and only if they are equivalent when computed on such elements.
	
	\noindent Consider then an element $(T,\alpha)$, with $T$ a tree, and let $c$ be an object of $P$. By inspection, we check that there is an isomorphism $$ (\rho_!^P)_\infty(T,\alpha)(c)\simeq \cN(\alpha/c)\simeq \mathsf{Env}(T)\times_{\mathsf{Env}(P)}{\mathsf{Env}(P)}_{/c},$$ and the latter space is equivalent, by definition, to the space given by $\mathsf{St}^P(T,\alpha)(c)$. The two $P$-algebras have therefore the same underlying family of spaces, and it is easy to check that the actions of $P$ are equivalent as well, so this concludes the argument.
\end{proof}

\begin{theorem}\label{QEstrunstr}
	Let $P$ be a $\Sigma$-free discrete operad. The pair $((\rho_!^P)_\infty, (\rho^*_P)_\infty)$ is an equivalence of $\infty$-categories. In particular, the pair of functors between $1$-categories $$ \adjunction{\rho_!^P}{\dsets/\cN_d P}{\alg_P(\ssets)}{\rho^*_P}$$ is a Quillen equivalence between the covariant and the projective model structure.
\end{theorem}

\begin{rmk}
	When $P$ is not $\Sigma$-free, it may happen that the Quillen adjunction $(\rho_!^P,\rho^*_P)$ is \emph{not} a Quillen equivalence. 
	
	\noindent Indeed, consider the case of $P=\mathsf{Com}$ the commutative operad, defined by $\mathsf{Com}(n)=\ast$ for all $n\geq 0$, with obvious composition law and symmetric group action. The dendroidal nerve of $\mathsf{Com}$ is the terminal object in the category of dendroidal sets, so that $$\dsets/ \cN \mathsf{Com} \simeq \dsets,$$ and the covariant model structure on the over category coincides with the so-called \emph{absolute covariant model structure} on $\dsets$, which models $\mathbb{E}_\infty$-algebras (\cite[Corollary 13.41]{HeMo:SDHT}).
	On the other hand, as detailed in \cite[Example 4.4]{W:MSCMGMC}, in the projective model structure on $\alg_{P}(\ssets)$ every path connected commutative algebra is equivalent to a generalized Eilenberg-Mac Lane space, i.e. product of Eilenberg-Mac Lane spaces. The existence of spaces like $QS=\Omega^\infty \Sigma^\infty \mathbb{S}^0$, which has an $\mathbb{E}_\infty$-algebra structure but is not a generalized Eilenberg-Mac Lane space, demonstrates that the Quillen adjunction cannot be a Quillen equivalence.
\end{rmk}

\begin{rmk}
	In \cite{B:RCPCCDA}, Barata defines the symmetric monoidal envelope of a dendroidal $\infty$-operad as a functor $\mathsf{d}\env(-)^\otimes\colon \dsets\to \ssets/\mathsf{Fin}_*$ and proves that when applied to a discrete operad it agrees with the symmetric monoidal envelope functor. The homotopical properties of $\mathsf{d}\env^\otimes$ have not been studied yet, but whether it is possible to use the dendroidal envelope to express the rectification functor is subject of future investigations.
\end{rmk}

\addtocontents{toc}{\SkipTocEntry}

\providecommand{\bysame}{\leavevmode\hbox to3em{\hrulefill}\thinspace}
\providecommand{\MR}{\relax\ifhmode\unskip\space\fi M`R }
\providecommand{\MRhref}[2]{%
	\href{http://www.ams.org/mathscinet-getitem?mr=#1}{#2}}
\providecommand{\href}[2]{#2}

\bibliographystyle{alpha}

\end{document}